\documentclass[reqno,a4paper,12pt]{amsart}
\usepackage{a4wide}

\usepackage{amsmath,amssymb,amscd,xcolor,amsthm,graphicx,enumerate,float}
\usepackage{todonotes}

\usepackage{graphicx}

\usepackage{hyperref}
\hypersetup{breaklinks=true}

\newcommand{\Balpha}{\mbox{$\hspace{0.1em}\rule[0.01em]{0.05em}{0.39em}\hspace{-0.21em}\alpha$}}

\numberwithin{equation}{section}
\theoremstyle{plain}

\makeatletter
\newtheorem*{rep@theorem}{\rep@title}
\newcommand{\newreptheorem}[2]{%
\newenvironment{rep#1}[1]{%
 \def\rep@title{#2 \ref{##1}}%
 \begin{rep@theorem}}%
 {\end{rep@theorem}}}
\makeatother

\newtheorem{theorem}[equation]{Theorem}
\newreptheorem{theorem}{Theorem}

\newtheorem{proposition}[equation]{Proposition}
\newtheorem{lemma}[equation]{Lemma}
\newtheorem{corollary}[equation]{Corollary}
\newtheorem*{conjecture}{Conjecture}
\newtheorem{claim}[equation]{Claim}

\newtheorem*{maintheorem}{Main Theorem}
\theoremstyle{remark}
\newtheorem{remark}[equation]{Remark}

\theoremstyle{definition}
\newtheorem{definition}[equation]{Definition}

\newtheorem*{question*}{Question}

\newcommand{\A}{{\mathcal A}}

\newcommand{\K}{{\mathcal K}}

\newcommand{\R}{\mathbb R}

\newcommand{\spt}{\operatorname{spt}}
\newcommand{\graph}{\operatorname{graph}}

\newcommand{\al}{\alpha}

\newcommand{\D}{\partial}
\newcommand{\de}{\delta}

\newcommand{\eps}{\varepsilon}

\newcommand{\ol}{\overline}

\providecommand{\abs}[1]{\lvert #1\rvert}

\def\XXint#1#2#3{{\setbox0=\hbox{$#1{#2#3}{\int}$}
     \vcenter{\hbox{$#2#3$}}\kern-.5\wd0}}

\begin{document}

\title[The moduli space of 2-convex embedded spheres]{The moduli space of two-convex\\ embedded spheres}
\author{Reto Buzano \and Robert Haslhofer \and Or Hershkovits}

\date{\today \thanks{The authors have been supported by EPSRC grant EP/M011224/1, NSERC grant RGPIN-2016-04331, and NSF grant DMS-1406394. We thank Richard Bamler and Igor Belegradek for helpful discussions.}}

\begin{abstract}
We prove that the moduli space of $2$-convex embedded $n$-spheres in $\mathbb{R}^{n+1}$ is path-connected for every $n$. Our proof uses mean curvature flow with surgery and can be seen as an extrinsic analog to Marques' influential proof of the path-connectedness of the moduli space of positive scalar curvature metics on three-manifolds \cite{Marques}.
\end{abstract}

\maketitle

\tableofcontents

\section{Introduction}

To put things into context, let us start with a general discussion of the moduli space of embedded $n$-spheres in $\mathbb{R}^{n+1}$, i.e. the space
\begin{equation}
\mathcal{M}_n = \textrm{Emb}(S^n,\mathbb{R}^{n+1})/\textrm{Diff}(S^n)
\end{equation}
equipped with the smooth topology.

In 1959, Smale proved that the space of embedded circles in the plane is contractible \cite{Smale}, i.e.
\begin{equation}\label{smale_thm}
\mathcal{M}_1\simeq \ast .
\end{equation}
In particular, the assertion $\pi_0(\mathcal{M}_1)=0$ is equivalent to the smooth version of the Jordan-Schoenflies theorem, and the assertion  $\pi_1(\mathcal{M}_1)=0$ is equivalent to Munkres' theorem that $\textrm{Diff}_+(S^2)$ is path-connected \cite{Munkres}.

Moving to $n=2$, Smale conjectured that the space of embedded 2-spheres in $\mathbb{R}^3$ is also contractible, i.e. that
\begin{equation}
\mathcal{M}_2\simeq \ast .
\end{equation}
In 1983, Hatcher proved the Smale conjecture \cite{Hatcher}. To understand the full strength of this theorem, let us again discuss the special cases $\pi_0$ and $\pi_1$. The assertion $\pi_0(\mathcal{M}_2)=0$ is equivalent to Alexander's strong form of the three dimensional Schoenflies theorem \cite{Alexander}. The assertion $\pi_1(\mathcal{M}_2)=0$ is equivalent to Cerf's theorem that $\textrm{Diff}_+(S^3)$ is path-connected \cite{Cerf}, which had wide implications in differential topology (and in particular has the important corollary that $\Gamma_4=0$).

For $n\geq 3$ not a single homotopy group of $\mathcal{M}_n$ is known. The case $n=3$ is related to major open problems in $4$-manifold topology. For example, finding a nontrivial element in $\pi_0( \mathcal{M}_3)$ would give a counterexample to the Schoenflies conjecture, and thus in particular a counterexample to the smooth 4d Poincar\'e conjecture. Regardless of the answer to these major open questions, we definitely have:
\begin{equation}
\textrm{The naive guess that } \mathcal{M}_n\simeq\ast  \textrm{ for all $n$ is completely false}.
\end{equation}
Indeed, if $\mathcal{M}_n$ were contractible for every $n$, then arguing as in the appendix of \cite{Hatcher}, we could infer that $\mathcal{D}_n:=\textrm{Diff}(D^{n+1}\, \textrm{rel}\, \partial D^{n+1})$ is contractible for every $n$. However, it is known that $\mathcal{D}_n$ has non-vanishing homotopy groups for every $n\geq 4$ \cite{CrowleySchick}.

In view of the topological complexity of $\mathcal{M}_n$ for $n\geq 3$,  it is an interesting question whether one can still derive some positive results on the space of embedded $n$-spheres under some curvature conditions. Such results would show that all the non-trivial topology of $\mathcal{M}_n$ is caused by embeddings of $S^n$ that are geometrically very far away from the canonical one. Some motivation to expect such an interaction between geometry and topology comes from what is known about the differential topology of the sphere itself. Namely, despite the existence of exotic spheres in higher dimensions \cite{Milnor}, it was proved by Brendle-Schoen that every sphere which admits a metric with quarter-pinched sectional curvature must be standard \cite{BrendleSchoen}. 

Motivated by the topological classification result from \cite{HuiskenSinestrari}, we consider $2$-convex embeddings, i.e. embeddings such that the sum of the smallest two principal curvatures is positive. Clearly, $2$-convexity is preserved under reparametrizations. We can thus consider the subspace
\begin{equation}
\mathcal{M}_n^{\textrm{2-conv}}\subset \mathcal{M}_n
\end{equation}
of $2$-convex embedded $n$-spheres in $\mathbb{R}^{n+1}$. We propose the following higher dimensional Smale type conjecture.

\begin{conjecture}The moduli space of 2-convex embedded $n$-spheres in $\R^{n+1}$ is contractible for every dimension $n$, i.e.
\begin{equation}
\mathcal{M}_n^{\textrm{2-conv}} \simeq\ast .
\end{equation}
\end{conjecture}

In the present article, we confirm the $\pi_0$-part of this conjecture.

\begin{maintheorem}\label{main_thm}
The moduli space of 2-convex embedded $n$-spheres in $\R^{n+1}$ is path-connected for every dimension $n$, i.e.
\begin{equation}
\pi_0(\mathcal{M}_n^{\textrm{2-conv}}) = 0.
\end{equation}
\end{maintheorem}

To the best of our knowledge, our main theorem is the first positive topological result about a moduli space of embedded spheres in any dimension $n\geq 3$ (except of course for the moduli space of convex embedded spheres, which is easily seen to be contractible).

Our main theorem can be seen as an extrinsic analog to the path-connectedness of the moduli space of positive scalar curvature metics in dimension two and three. More precisely, let $M^n$ be a compact orientable $n$-dimensional manifold and denote by $\mathcal{R}_+(M^n)$ the set of Riemannian metrics on $M^n$ with positive scalar curvature. In \cite{Marques}, Marques proved that for $n=3$, if $\mathcal{R}_+(M^3) \neq \emptyset$, the moduli space $\mathcal{R}_+(M^3)/\textrm{Diff}(M^3)$ is path-connected, i.e.
\begin{equation}\label{Marques}
\pi_0(\mathcal{R}_+(M^3)/\textrm{Diff}(M^3)) = 0.
\end{equation}
Combining this with Cerf's result \cite{Cerf} that $\textrm{Diff}_+(S^3)$ is path-connected, as mentioned above, Marques also finds the corollary that the space $\mathcal{R}_+(S^3)$ of positive scalar curvature metrics on the three-sphere is path-connected. While the lower dimensional statement that $\mathcal{R}_+(S^2)$ is path-connected is a classical theorem of Weyl \cite{Weyl}, Marques' result is surprising because in higher dimensions the picture is very different. Indeed, for $n=3+4k$ with $k\in\mathbb{N}$, both $\mathcal{R}_+(S^n)$ and $\mathcal{R}_+(S^n)/\textrm{Diff}(S^n)$ have infinitely many path-connected components \cite{Carr, KreckStolz}.

Marques' fascinating proof \cite{Marques} uses powerful methods from geometric analysis, most importantly the theory of three-dimensional Ricci flow with surgery developed by Perelman \cite{perelman_entropy,perelman_surgery}. Inspired by his work, our proof of the main theorem uses a similar scheme for the mean curvature flow with surgery.

To illustrate the idea in a simplified setting, let us start by sketching a geometric-analytic proof of Smale's theorem, i.e. of assertion \eqref{smale_thm}. For $n=1$, the mean curvature flow is usually called the curve shortening flow and it deforms a given smooth closed embedded curve $\Gamma \subset \mathbb{R}^2$ in time with normal velocity given by the curvature vector. This produces a one-parameter family of curves $\{\Gamma_t\}_{t\in [0,T)}$ with $\Gamma_0=\Gamma$. By a beautiful theorem of Grayson \cite{Grayson}, the flow becomes extinct in a point, and after suitably rescaling converges to a round circle. If $A_\Gamma$ denotes the enclosed area, then the extinction time can be explicitly computed by the simple formula $T_\Gamma=A_\Gamma/2\pi$. Let $x_\Gamma\in \mathbb{R}^2$ denote the extinction point and consider the map $\mathcal{I}:\mathcal{M}_1\times [0,1]\to \mathcal{M}_1$ defined for $t<1$ by
\begin{equation}
\mathcal{I}(\Gamma,t)= \sqrt{\frac{1-t(1-\pi/A_\Gamma)}{1-t}}\cdot\left ( \Gamma_{\frac{A_\Gamma}{2\pi}t}-x_\Gamma\right)+(1-t)x_\Gamma.
\end{equation}
Then $\mathcal{I}(\Gamma,0)=\Gamma$, $\lim_{t\to 1}\mathcal{I}(\Gamma,t)=S^1$ by Grayson's theorem, and $\mathcal{I}$ is continuous if we set $\mathcal{I}(\Gamma,1)=S^1$. Thus $\mathcal{I}$ is a homotopy proving that $\mathcal{M}_1\simeq \ast$.

For $n\geq 2$, the study of mean curvature flow is much more involved due to the formation of local singularities. For example, if the initial condition looks like a dumbbell, then a neck-pinch singularity develops \cite{ref_neckpinch}. In this case, one obtains a round shrinking cylinder as blowup limit. Another example of a singularity is the degenerate neck-pinch \cite{ref_degneckpinch}, which is modelled on a self-similarly translating soliton. While in general there are highly complicated singularity models  for the mean curvature flow (see e.g. \cite{KKM}), the above two examples essentially illustrate what singularities can occur if the initial hypersurface is $2$-convex. Namely, in this case the only self-similarly shrinking singularity models are the round sphere $S^n$ and the round cylinder $S^{n-1}\times \mathbb{R}$ \cite{Huisken_local_global,HuiskenSinestrari_convexity}, and the only self-similarly translating singularity model is the rotationally symmetric bowl soliton \cite{Haslhofer_bowl}.

The above picture suggests a surgery approach for the mean curvature flow of $2$-convex hypersurfaces, where one heals local singularities by cutting along cylindrical necks and gluing in standard caps.\footnote{In the case of the degenerate neck-pinch one can find a neck at controlled distance from the tip.} This idea has been implemented successfully in the deep work of Huisken-Sinestrari \cite{HuiskenSinestrari}, and in more recent work by Haslhofer-Kleiner \cite{HK_surgery} and Brendle-Huisken \cite{BrendleHuisken}. In a flow with surgery, there are finitely many times where suitable necks are replaced by standard caps and/or where connected components with specific geometry and topology are discarded. Away from these finitely many times, the evolution is simply smooth evolution by mean curvature flow. By the above quoted articles, mean curvature flow with surgery exists for every $2$-convex initial condition where $2$-convexity is preserved under the flow and the surgery procedure. In this paper, we use the approach from Haslhofer-Kleiner \cite{HK_surgery}, which has the advantage that it works in every dimension and that it comes with a canonical neighborhood theorem \cite[Thm.~1.22]{HK_surgery} that is quite crucial for our topological application.

Let us now sketch the key ideas of the proof of our main theorem.

Given a $2$-convex embedded sphere $M_0\subset \mathbb{R}^{n+1}$, we consider its mean curvature flow with surgery $\{M_t\}_{t\in [0,\infty)}$ as provided by the existence theorem from \cite[Thm. 1.21]{HK_surgery}. The flow always becomes extinct in finite time $T<\infty$. The simplest possible case is when the evolution is just smooth evolution by mean curvature flow and becomes extinct (almost) roundly. In this case, a rescaled version of $\{M_t\}_{t\in [0,T)}$ yields the desired path in $\mathcal{M}_n^{\textrm{2-conv}}$ connecting $M_0$ to a round sphere. In general, things are more complicated since (a) there can be surgeries and/or discarding of some components, and (b) the hypersurface doesn't necessarily become round.

Let us first outline the analysis of the discarded components.
By the canonical neighborhood theorem \cite[Thm. 1.22]{HK_surgery} and the topological assumption on $M_0$, each connected component which is discarded is either a convex sphere of controlled geometry or a capped-off chain of $\eps$-necks. This information is sufficient to construct an explicit path in $\mathcal{M}_n^{\textrm{2-conv}}$ connecting any discarded component to a round sphere (see Section \ref{sec_discarded}).

While surgeries disconnect the hypersurfaces into different connected components, for our topological application we eventually have to glue the pieces together again. To this end, in Section \ref{sec_gluing}, we prove the existence of a connected sum operation that preserves $2$-convexity and embeddedness, and which depends continuously on the gluing configuration. The connected sum glues together 2-convex spheres along what we call strings, and also allows for capped-off strings. A string is a tiny tube around an admissible curve. The mean curvature flow with surgery comes with several parameters, in particular the trigger radius $H_{\textrm{trig}}^{-1}$. The point is to choose the string radius $r_s\ll H_{\textrm{trig}}^{-1}$ to ensure that the different scales in the problem barely interact.

We then argue by backwards induction on the surgery times using a scheme of proof which is inspired by related work on Ricci flow with surgery and moduli spaces of Riemannian metrics \cite{perelman_entropy,perelman_surgery,Marques}. The main claim we prove by backwards induction is that at each time every connected component is isotopic via 2-convex embeddings to what we call a marble tree. Roughly speaking, a marble tree is a connected sum of small spheres (marbles) along admissible strings that does not contain any loop. At the extinction time, by the above discussion, every connected component is isotopic via 2-convex embeddings to a round sphere, i.e. a marble tree with just a single marble and no strings. The key ingredient for the induction step is the connected sum operation which, as we will explain, allows us to glue isotopies, see Section \ref{sec_proofmain}. This is illustrated in Figure \ref{fig.1}. A final key step, which we carry out in Section \ref{sec_marble}, is to prove that every marble tree is isotopic via 2-convex embeddings to a round sphere. This concludes the outline of the proof of the main theorem.

Let us compare our proof to Marques' proof of \eqref{Marques} from \cite{Marques}. First, Marques introduces the concept of \emph{canonical metrics}, which are obtained from the standard round three-sphere by attaching to it finitely many spherical space forms via a connected sum construction and adding finitely many handles (via connected sums of the sphere to itself). It is easy to see that the space of canonical metrics on a fixed manifold $M^3$ is path-connected. Marques then uses Ricci flow with surgery, backwards induction, and connected sums to define a continuous path from any given positive scalar curvature metric to a canonical metric. Our approach described above is an extrinsic variant of his scheme of proof with marble trees playing a similar role to his canonical metrics. Several steps of Marques' work rely heavily on a connected sum construction which preserves positive scalar curvature. This was introduced independently by Gromov-Lawson \cite{GL} and Schoen-Yau \cite{SY}. In our context, a suitable connected sum construction for hypersurfaces preserving two-convexity and embeddedness was not available in the existing literature, and so the bulk of the technical work in this paper is devoted to developing such a construction. This may be of independent interest. Another difference is that the conformal method is not available in extrinsic geometry. As a substitute, we construct some explicit isotopies from our marble trees to round spheres. Finally, and more technically, in our extrinsic setting it is important that we keep track of where surgery procedures happen in ambient space. In fact, it can become necessary to move strings out of surgery regions finitely many times during the backwards induction process -- we explain this in detail in Section \ref{sec_proofmain}.

\begin{figure}[H]
\includegraphics[width=10cm]{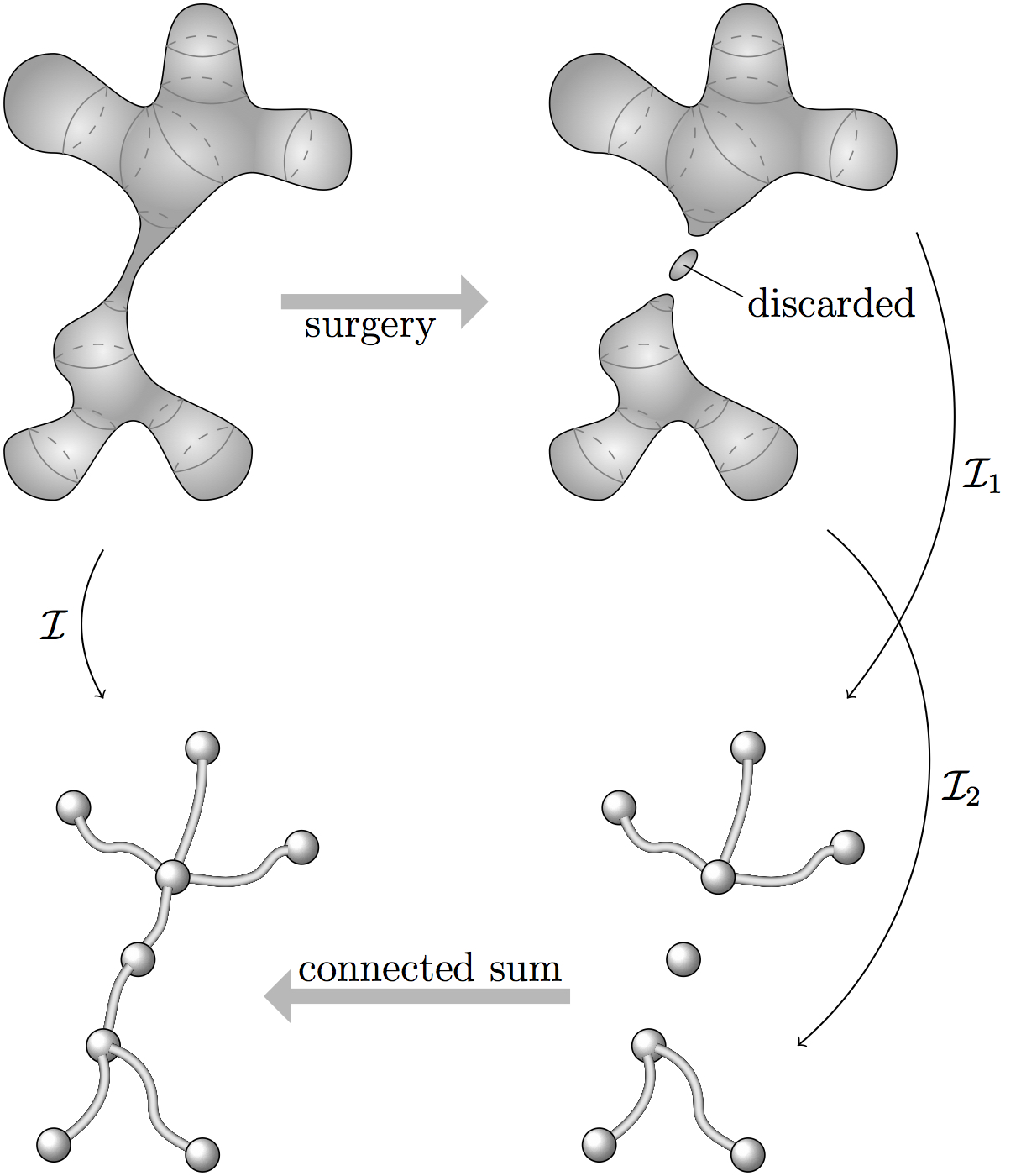}
\vspace{5mm}
\caption{By inductive assumption, each connected component of the post-surgery manifold and every discarded component is isotopic to a marble tree. We will prove that these isotopies can be glued together to yield an isotopy of the pre-surgery manifold to a connected sum of the marble trees.}\label{fig.1}
\end{figure}
\vspace{-2mm}

This article is organized as follows. In Section \ref{sec_basic}, we introduce some basic notions, such as controlled configurations of domains and curves, and capped-off tubes. In Section \ref{sec_transition}, we construct some transition functions satisfying certain differential inequalities which are essential to preserve $2$-convexity under gluing. In Section \ref{sec_gluing}, we prove the existence of a suitable connected sum operation. In Section \ref{sec_marble}, we prove that marble-trees can be deformed to a round sphere preserving $2$-convexity and embeddedness. In Section \ref{sec_necksurgery}, we recall the relevant notions of necks and surgery and construct an isotopy connecting the pre-surgery manifold with a modified post-surgery manifold containing a glued-in string for each neck that was cut out. In Section \ref{sec_discarded}, we construct the isotopies for the discarded components. In Section \ref{sec_mcfsurgery}, we summarize some aspects of mean curvature flow with surgery from \cite{HK_surgery} that are needed for the proof of the main theorem. Finally, in Section \ref{sec_proofmain}, we combine all the ingredients and prove the main theorem.

\section{Basic notions, controlled configurations and caps}\label{sec_basic}

In this section, we first fix some basic notions, and then define controlled configurations of curves and domains, and capped-off tubes.

Instead of talking about a closed embedded hypersurface $M\subset \mathbb{R}^{n+1}$ it is slightly more convenient for us to talk about the compact domain $K\subset\mathbb{R}^{n+1}$ with $\partial K=M$. This point of view is of course equivalent, since $M$ determines $K$ and vice versa. A domain $K\subset\mathbb{R}^{n+1}$ is called \emph{$2$-convex}, if
\begin{equation}
\lambda_1+\lambda_2>0
\end{equation}
for all $p\in\partial K$, where $\lambda_1\leq \lambda_2\leq \ldots \leq \lambda_n$ are the principal curvatures, i.e. the eigenvalues of the second fundamental form $A$ of $\partial K$. In particular, every $2$-convex domain has positive mean curvature
\begin{equation}
H>0.
\end{equation}

The isotopies we construct in this paper can be described most conveniently via smooth families $\{K_t\}_{t\in [0,1]}$ of 2-convex domains (note that talking about domains has the advantage that it automatically incorporates embeddedness). We say that an isotopy $\{K_t\}_{t\in [0,1]}$ is \emph{trivial outside a set $X$} if $K_t\cap (\R^{n+1}\setminus X)$ is independent of $t$.
We say that an isotopy is \emph{monotone} if $K_{t_2}\subseteq K_{t_1}$ for $t_2\geq t_1$, and \emph{monotone outside a set $X$} if $K_{t_2}\cap (\R^{n+1}\setminus X)\subseteq K_{t_1}\cap (\R^{n+1}\setminus X)$ for $t_2\geq t_1$.

We also need more quantitative notions of $2$-convexity and embeddedness. A $2$-convex domain is called \emph{$\beta$-uniformly $2$-convex}, if
\begin{equation}
\lambda_1+\lambda_2\geq \beta H
\end{equation}
for all $p\in\partial K$. A domain $K\subset\mathbb{R}^{n+1}$ is called \emph{$\alpha$-noncollapsed} (see \cite{sheng_wang,andrews1,HK}) if it has positive mean curvature and each boundary point $p\in \partial K$ admits interior and exterior balls tangent at $p$ of radius at least $\alpha/H(p)$.

\begin{definition}[{$\mathbb{A}$-controlled domains}]\label{def_contr_dom}
Let $\alpha\in (0,n-1)$, $\beta\in (0,\tfrac{1}{n-1})$, $c_H>0$, and let $C_A,D_A<\infty$.
A domain $K\subset\mathbb{R}^{n+1}$ is called \emph{$(\alpha,\beta,c_H,C_A,D_A)$-controlled}, if it is $\alpha$-noncollapsed, $\beta$-uniformly $2$-convex, and satisfies the bounds
\begin{equation}\label{eq_bound_contr_dom}
H\geq c_H ,\qquad \lambda_n\leq C_A ,\qquad \abs{\nabla A}\leq D_A.
\end{equation}
We write $\mathbb{A}=(\alpha,\beta,c_H,C_A,D_A)$ to keep track of the constants. Moreover, we introduce the following three classes of domains.
\begin{enumerate}[\hspace{3mm}(1)]
\item The set $\mathcal{D}$ of all (possibly disconnected) 2-convex smooth compact domains $K\subset \mathbb{R}^{n+1}$.
\item The set $\mathcal{D}_{\mathbb{A}}=\{K\in\mathcal{D}\;|\;\textrm{$K$ is $\mathbb{A}$-controlled} \}$ of all $\mathbb{A}$-controlled domains.
\item The set $\mathcal{D}_{\mathbb{A}}^{\textrm{pt}} =\{(K,p)\in \mathcal{D}_{\mathbb{A}}\times\mathbb{R}^{n+1} \; |\; p\in \partial K\}$ of all pointed $\mathbb{A}$-controlled domains.
\end{enumerate}
\end{definition}

To facilitate the gluing construction, let us now introduce a notion of controlled configurations of domains and curves.

\begin{definition}[{$b$-controlled curves}]
Let $b>0$. An oriented compact curve $\gamma\subset \mathbb{R}^{n+1}$ (possibly with finitely many components) is called \emph{$b$-controlled} if the following conditions are satisfied.
\begin{enumerate}[\hspace{3mm}(a)]
\item The curvature vector satisfies $|\kappa|\leq b^{-1}$ and $|\partial_s \kappa|\leq b^{-2}$.
\item Each connected component has normal injectivity radius $\geq \frac{1}{10}b$.
\item Different connected components are distance $\geq 10 b$ apart.  
\end{enumerate} 
Moreover, we introduce the following two classes of curves.
\begin{enumerate}[\hspace{3mm}(1)]
\item The set $\mathcal{C}_{b}$ of all $b$-controlled curves $\gamma\subset\mathbb{R}^{n+1}$.
\item The set $\mathcal{C}^{\textrm{pt}}_{b}=\{(\gamma,p)\in \mathcal{C}_{b}\times \mathbb{R}^{n+1}\; | \; p\in\partial\gamma \, \}$ of pointed $b$-controlled curves.
\end{enumerate}
\end{definition}

\begin{definition}[Controlled configuration of domains and curves]\label{conrolled_conf}
We call a pair $(K,\gamma)\in \mathcal{D}_{\mathbb{A}}\times \mathcal{C}_{b} $ an \emph{$(\mathbb{A},b)$-controlled configuration} if 
\begin{enumerate}[\hspace{3mm}(a)]
\item The interior of $\gamma$ lies entirely in $\mathbb{R}^{n+1}\setminus K$.
\item The endpoints of $\gamma$ satisfy the following properties:
\begin{itemize}
\item If $p\in \partial\gamma\cap\partial K$, then $\gamma$ touches $\partial K$ orthogonally there.
\item If $p\in \partial\gamma\setminus\partial K$, then $d(p,\partial K)\geq 10b$.
\item $d \big( \gamma \setminus\bigcup_{p\in \partial\gamma}B_{b/10}(p),\partial K \big)\geq b/20$.
\end{itemize}
\end{enumerate}
Let $\mathcal{X}_{\mathbb{A},b}\subseteq \mathcal{D}_{\mathbb{A}}\times \mathcal{C}_{b}$ be the set of all $(\mathbb{A},b)$-controlled configurations.
\end{definition}

A map $\mathcal{F}:\mathcal{X}_{\mathbb{A},b}\rightarrow \mathcal{D}$ is called \emph{smooth} if for every $k\in \mathbb{N}$ and every smooth family $\phi:B^k\rightarrow \mathcal{X}_{\mathbb{A},b}$ (in the parametrized sense), $\mathcal{F}\circ \phi:B^k \rightarrow \mathcal{D}$ is a smooth family. Smoothness of maps between other combinations of the above spaces is defined similarly.

\begin{definition}[{Standard cap}]\label{def_stdcap}
A \emph{standard cap} is a smooth convex domain $K^{\textrm{st}}\subset \R^{n+1}$ such that
\begin{enumerate}[\hspace{3mm}(a)]
\item $K^{\textrm{st}}\cap \{x_1>0.01\}=B_1(0)\cap \{x_1>0.01\}$.
\item $K^{\textrm{st}}\cap \{x_1<-0.01\}$ is a solid round half-cylinder of radius $1$.
\item $K^{\textrm{st}}$ is given by revolution of a function $u^{\textrm{st}}:(-\infty,1]\rightarrow \mathbb{R}_+$.
\end{enumerate} 
\end{definition}

Definition \ref{def_stdcap} is consistent with \cite[Def 2.2, Prop 3.10]{HK_surgery}. For given $\alpha$ and $\beta$, we now fix a suitable standard cap $K^{\textrm{st}}=K^{\textrm{st}}(\alpha,\beta)$ which is $\alpha$-noncollapsed and $\beta$-uniformly $2$-convex.
 
\begin{definition}[Capped-off tube]\label{cap_off_def}
Let $\gamma:[a_0,a_1]\rightarrow \mathbb{R}^{n+1}$ be a connected $b$-controlled curve parametrized by arc-length.
\begin{enumerate}[\hspace{3mm}(1)]
\item A \emph{right capped-off tube} of radius $r<b/10$ around $\gamma$ at a point $p=\gamma(s_0)$, where $s_0\in [a_0+r,a_1]$, is the domain 
\begin{align}
CN^+_{r}(\gamma,p)=
\{\gamma(s)+ u^{\textrm{st}}\big(1-\tfrac{s_0-s}{r}\big)\!\cdot\!  \left(\bar{B}_r^{n+1}\cap \gamma'(s)^{\perp}\right) |\;s\in [a_0,a_1]\} \nonumber.
\end{align}
\item Similarly, a \emph{left capped-off tube} of radius $r<b/10$ around $\gamma$ at a point $p=\gamma(s_0)$, where $s_0\in [a_0,a_1-r]$, is the domain 
\begin{align}
CN^-_{r}(\gamma,p)=
\{\gamma(s)+ u^{\textrm{st}}\big(1-\tfrac{s-s_0}{r}\big)\!\cdot\! \left(\bar{B}_r^{n+1}\cap \gamma'(s)^{\perp}\right) |\;s\in [a_0,a_1]\} \nonumber.
\end{align}
\end{enumerate}  
\end{definition}

\begin{proposition}[Capped-off tubular neighborhoods]\label{mean_convex_capped_off}
There exists a constant $\zeta=\zeta(\alpha,\beta)>0$ such that for every $r<\zeta b$, every connected $b$-controlled curve $\gamma$, and every $p\in\gamma$ with distance at least $r$ from the right/left endpoint the capped-off tubular neighborhood $CN^\pm_{r} (\gamma,p)$ is $\beta$-uniformly $2$-convex.  
\end{proposition}
\begin{proof}
Suppose there are sequences $\gamma_i$ of $b_i$-controlled curves and points $p_i$ as above such that $CN^\pm_{i^{-1}b_i}(\gamma_i,p_i)$ is not $\beta$-uniformly 2-convex at a point $q_i\in\partial CN^\pm_{i^{-1}b_i}(\gamma_i,p_i)$. Moving $q_i$ to the origin and scaling by $ib_i^{-1}$ we get  sub-convergence to either a cylinder or a capped of cylinder, both of which satisfy $\lambda_1+\lambda_2>\beta H$; this is a contradiction.
\end{proof}

\section{Transition functions and gluing principle}\label{sec_transition}

In this section, we construct some transition functions that will be needed later. We also explain a gluing principle that will be used frequently.

Fix a smooth function $\chi:\mathbb{R}\to[0,1]$ with $\spt({\chi'})\subset (0,1)$ such that $\chi(t)= 0$ for $t\leq 0$, $\chi(t)= 1$ for $t\geq1$, $\abs{\chi'} \leq 10$ and $|\chi''| \leq 100$. For $\eps>0$, we then set
\begin{equation}\label{eq_stepfn}
\chi_\eps(t)=\chi(t/\eps).
\end{equation}
\begin{proposition}[Transition function]\label{good_transition}
There exists a smooth function $\varphi:\mathbb{R} \rightarrow \mathbb{R}$ with $\spt({\varphi'})\subset (0,1)$ such that $\varphi(t)=-\frac{1}{2}$ for $t\leq 0$, $\varphi(t)=\frac{1}{4}$ for $t\geq 1$, and 
$Q[\varphi]:=t^2\varphi''+4t\varphi'+2\varphi\leq 0.95$ for all $t\in\mathbb{R}$.
\end{proposition}

To construct $\varphi$, we fix $\delta>0$ small, and start by solving the initial value problem
\begin{equation}
Q[\tilde{\varphi}]=\tfrac12,\qquad \tilde{\varphi}(\delta)=-\tfrac{1}{2},\, \tilde{\varphi}'(\delta)=0.
\end{equation}
The solution is given by
\begin{equation}\label{sol_of_ivp}
\tilde{\varphi}(t)=\frac{1}{4}-\frac{3}{2}\frac{\delta}{t}+\frac{3}{4}\frac{\delta^2}{t^2}.
\end{equation}

At $t=\delta$ this fits together with the constant function $\varphi\equiv-\frac{1}{2}$ up to first order. However, since $\tilde{\varphi}''(\delta)=\frac{3}{2\delta^2}\neq 0$, we have to adjust the derivatives of order two and higher to make the transition smooth.

\begin{lemma}[Higher derivatives transition iteration step]\label{Q_transition_step}
If $\varphi_0,\varphi_1$ are smooth functions satisfying $Q[\varphi_i]\leq 0.9$ in a neighborhood of $t_\ast$ and
\begin{enumerate}[\hspace{3mm}(1)]
\item $\varphi_{0}(t_\ast)=\varphi_{1}(t_\ast)$, $\varphi'_{0}(t_\ast)=\varphi'_{1}(t_\ast)$,
\item $|t_\ast^2\varphi_0''(t_\ast)-t_\ast^2\varphi_1''(t_\ast)|\leq 10^{-5}$,
 \end{enumerate}  
then, for every  $\eps>0$, there exists a smooth function $\varphi$ satisfying $Q[\varphi]\leq 0.95$ such that $\varphi\equiv \varphi_0$ for $t\leq t_\ast$ and $\varphi\equiv \varphi_1$ for $t\geq t_\ast+\eps$.
\end{lemma}
\begin{proof}
It suffices to prove the assertion for  $\eps$ as small as we want. In particular, we can assume that $\eps<t_\ast/2$ and that $\varphi_0$ and $\varphi_1$ are defined on $[t_\ast,t_\ast+2\eps)$. Let $D<\infty$ be an upper bound for $|\varphi_0'''(t)|+|\varphi_1'''(t)|$ for $t\in[t_\ast,t_\ast+\eps]$.
Set
\begin{equation*}
\varphi(t)=(1-\chi_\eps(t-t_\ast))\varphi_0(t)+\chi_\eps(t-t_\ast) \varphi_1(t),
\end{equation*}
where $\chi_\eps$ is the step function from \eqref{eq_stepfn}. Then
\begin{align*}
Q[\varphi](t) &=(1-\chi_\eps(t-t_\ast))Q[\varphi_0](t)+\chi_\eps(t-t_\ast) Q[\varphi_1](t)\\
&\quad+t^2\big(\chi_\eps''(t-t_\ast)(\varphi_1-\varphi_0)(t)+2\chi_\eps'(t-t_\ast)(\varphi_1-\varphi_0)'(t)\big)\\
&\quad+4t\chi_\eps'(t-t_\ast)(\varphi_1-\varphi_0)(t).
\end{align*} 
All terms in the second and third line can be made small, by taking $\eps$ small enough (depending on $t_\ast$ and $D$). For instance,
\begin{equation}
|t^2\chi_\eps''(t-t_\ast)(\varphi_1-\varphi_0)(t)| \leq t^2 \frac{100}{\eps^2}\left(\frac{10^{-5}}{2t_\ast^2}\eps^2+D\eps^3\right)\leq 10^{-2}
\end{equation}
for $\eps<\min\{t_\ast/2,10^{-5} t_\ast^{-2}D^{-1}\}$, and similarly for the other terms. Since we also have $(1-\chi_\eps(t-t_\ast))Q[\varphi_0](t)+\chi_\eps(t-t_\ast) Q[\varphi_1](t)\leq 0.9$, this implies the assertion.
\end{proof}

\begin{proof}[Proof of Proposition \ref{good_transition}]
We start with the function $\tilde{\varphi}$ from \eqref{sol_of_ivp}. It satisfies $\tilde{\varphi}''(\delta)=\frac{3}{2\delta^2}$ and $|\tilde{\varphi}'''| \lesssim \delta^{-3}$ on $[\delta/2,\delta]$. Then, by applying Lemma \ref{Q_transition_step} repeatedly (say $2\cdot 10^{5}$ times) on disjoint intervals of length $\eps\ll 10^{-5}\cdot \delta$, we can interpolate $\tilde{\varphi}$ to create a function $\bar{\varphi}$ which is of the form $c_1+c_2t$ on $[0,\delta/2]$ for constant $c_1,c_2$ with $|c_1+\frac{1}{2}|<\delta$ and $|c_2|<\delta$ while maintaining $Q[\bar{\varphi}]\leq 0.95$.\footnote{The point is that applying Lemma \ref{Q_transition_step} at very small scales we can change $\varphi''$ by a definite amount with hardly changing $\varphi$ and $\varphi'$. This process can be iterated until $\varphi''$ reaches the desired value provided that $Q[\varphi]\leq 0.9$ near every gluing point.} It is clear that this can be interpolated to a constant $c_\ell$ with $|c_\ell+\frac{1}{2}|<2\delta$. Similarly (and more easily), since $|\tilde{\varphi}(\frac{1}{2})-\frac{1}{4}|,\max_{i=1}^{3}\{|\tilde{\varphi}^{(i)}(\frac{1}{2})|\}\lesssim \delta$ we can adjust the function  on the right end, so that $\bar{\varphi}\equiv c_r$ for $t\geq 3/4$, where $|c_r-\frac{1}{4}|<\delta$. Finally, the function $\varphi$ is obtained by slightly scaling $\bar{\varphi}$.   
\end{proof}

The above proof illustrates a general gluing principle that will  be used  frequently in later sections in closely related situations:

\begin{remark}[Gluing principle]\label{gluing_principle}
If we have two 1-variable functions that satisfy a second order differential inequality and fit together to first order, then we can adjust the second and higher derivative terms at very small scales such that we preserve the differential inequality and hardly change the zeroth and first order terms.
\end{remark}

We close this section with three more propositions, the spirit and proofs of which closely resemble what we have seen above.

\begin{proposition}[Trimming the remainder]\label{Third_order_trim}
Let $u$ be a smooth function satisfying $P[u]:=1+u'^2-uu''>0$, and let $u_2$ be its second order Taylor approximation at $t_{\ast}$. Then, for every $\eps>0$, there exists a smooth function $v$ (depending smoothly on $\eps$ and $t_\ast$) satisfying $P[v]>0$ for $t\leq t_\ast$, such that $v\equiv u$ for $t\leq t_\ast-\eps$ and $v\equiv u_2$ for $t\geq t_\ast$. 
\end{proposition}
\begin{proof}
Let $\chi_\eps$ be the step function from \eqref{eq_stepfn}, and let
\begin{equation}
v(t)=(1-\chi_{\eps}(t-t_\ast+\eps))u(t)+\chi_{\eps}(t-t_\ast+\eps)u_{2}(t). 
\end{equation}
Since $u$ and $u_2$ agree to second order at $t_{\ast}$, we have the estimate
\begin{equation}
\sup_{t\in[-\eps,0]}\abs{(v-u)^{(i)}(t_\ast+t)}\leq C \eps^{3-i}\qquad (i=0,1,2),
\end{equation}
where $C<\infty$ depends on a bound for $u$ and its first three derivatives on $[-\eps,0]$. Choosing $\eps$ small enough (depending on $C$ and $P[u]|_{t=t_\ast}$) the assertion follows.
\end{proof}
\begin{proposition}[Second transition function]\label{trasnsition_function_b}
For every $\eps>0$, there exists a smooth function $\psi:\mathbb{R} \rightarrow [0,1]$ with $\spt(\psi')\subset (0,1)$ and a constant $c>0$, such that $\psi(t)=1$ for $t\leq c$, $\psi(t)= 0$ for $t\geq 1$, and $|t^2\psi''|+|t\psi'|<\eps$. 
\end{proposition}
\begin{proof}
Let $N=\lceil 10000/\eps\rceil$. Then the function
\begin{equation*}
\psi(t):=
\begin{cases} 
      1 & t\leq 2^{-N} \\
      \tfrac{k-1}{N}+\tfrac{1}{N}\chi(2-2^{k}t) & t\in[2^{-k},2^{1-k}]\quad (k=1,\ldots, N)\\
      0 & t\geq 1
   \end{cases}
\end{equation*}
has the desired properties.
\end{proof}

\begin{proposition}\label{linear_interpolation}
Let $u_i(x)\in[0,1]$  ($i=0,1$) be two $C^2$-functions such that $u_i'' < 1$. Then setting $u_t(x)=tu_1(x)+(1-t)u_0(x)$ we have $P[u_t]:=1+(u_t')^2-u_tu_t''>0$.
\end{proposition}
\begin{proof}
This is clear.  
\end{proof}

\section{Attaching strings to 2-convex domains}\label{sec_gluing}

The goal of this section is to prove the existence of a gluing map, that preserves $2$-convexity, depends smoothly on parameters,  preserves symmetries, and is given by an explicit model in the case of round balls.

Recall that we denote by $\mathcal{X}_{\mathbb{A},b}$ the space of $(\mathbb{A},b)$-controlled configurations of domains and curves (see Definition \ref{conrolled_conf}), and by $\mathcal{D}$ the space of all 2-convex smooth compact domains (see Definition \ref{def_contr_dom}).

\begin{theorem}[Gluing map]\label{thm_attachstrings}
There exists a smooth, rigid motion equivariant map
\begin{equation*}
\mathcal{G}:\mathcal{X}_{\mathbb{A},b}\times (0,\bar{r})\rightarrow \mathcal{D},\quad ((D,\gamma),r)\mapsto \mathcal{G}_{r}(D,\gamma),
\end{equation*}
where $\bar{r}=\bar{r}(\mathbb{A},b)>0$ is a constant, and a smooth increasing function $\delta:(0,\bar{r})\rightarrow \mathbb{R}_+$ with $\lim_{r\rightarrow 0}\delta(r)=0$, with the following significance:
\begin{enumerate}[\hspace{3mm}(1)]
\item $\mathcal{G}_{r}(D,\gamma)$ deformation retracts to $D\cup \gamma$.
\item We have
\begin{equation*}
\qquad\quad\mathcal{G}_{r}(D,\gamma)\setminus\bigcup_{p\in \partial\gamma} B_{\delta(r)}(p) = D \cup  N_{r}(\gamma)\setminus \bigcup_{p\in \partial \gamma} B_{\delta(r)}(p),
\end{equation*}
where $N_{r}(\gamma)$ denotes the solid $r$-tubular neighborhood of $\gamma$. The collection of balls $\{B_{\delta(r)}(p)\}_{p\in\partial \gamma}$ is disjoint.
\item  If $p\in \partial \gamma\setminus\partial D$ and $\gamma_p$ denotes the connected component of $\gamma$ containing $p$ as its endpoint (see Definition \ref{conrolled_conf}), then
\begin{align*}
\mathcal{G}_{r}(D,\gamma)\cap B_{\delta(r)}(p)= 
CN_{r}(\gamma_p,p)\cap B_{\delta(r)}(p),  
\end{align*}
where $CN_{r}(\gamma_p,p)$ denotes the capped-off solid $r$-tube around $\gamma_p$ at $p$ (see Definition \ref{cap_off_def}). 
\item The construction is local: If $(D\cup\gamma)\cap B_{\delta(r)}(p)=(\tilde{D}\cup\tilde{\gamma})\cap B_{\delta(r)}(p)$ for some $p\in \partial\gamma$, then $\mathcal{G}_r(D,\gamma)\cap B_{\delta(r)}(p)=\mathcal{G}_r(\tilde{D},\tilde{\gamma})\cap B_{\delta(r)}(p)$.
\end{enumerate}
Moreover, in the special case that $D\cap B_{\delta(r)}(p)=\bar{B}_R(q)\cap B_{\delta(r)}(p)$ for some $q\in\mathbb{R}^{n+1}$, $R>r$, and if $\gamma \cap B_{\delta(r)}(p)$ is a straight line, we have
\begin{equation*}
\mathcal{G}_{r}(D,\gamma) \cap B_{\delta(r)}(p)=\left(R\cdot TC_{\varrho^{-1}(r/R)}+q\right) \cap B_{\delta(r)}(p),
\end{equation*}
where $T$ is an orthogonal transformation, and where the domain $C_\ast$ and the function $\varrho$ are from the explicit model in Proposition \ref{gluing_sphere_cylinder} below. 
\end{theorem}

We say that $\mathcal{G}_{r}(D,\gamma)$ is obtained from $D$ by \emph{attaching strings} of radius $r$ around $\gamma$.\\

\begin{figure}[H]
\includegraphics[width=12cm]{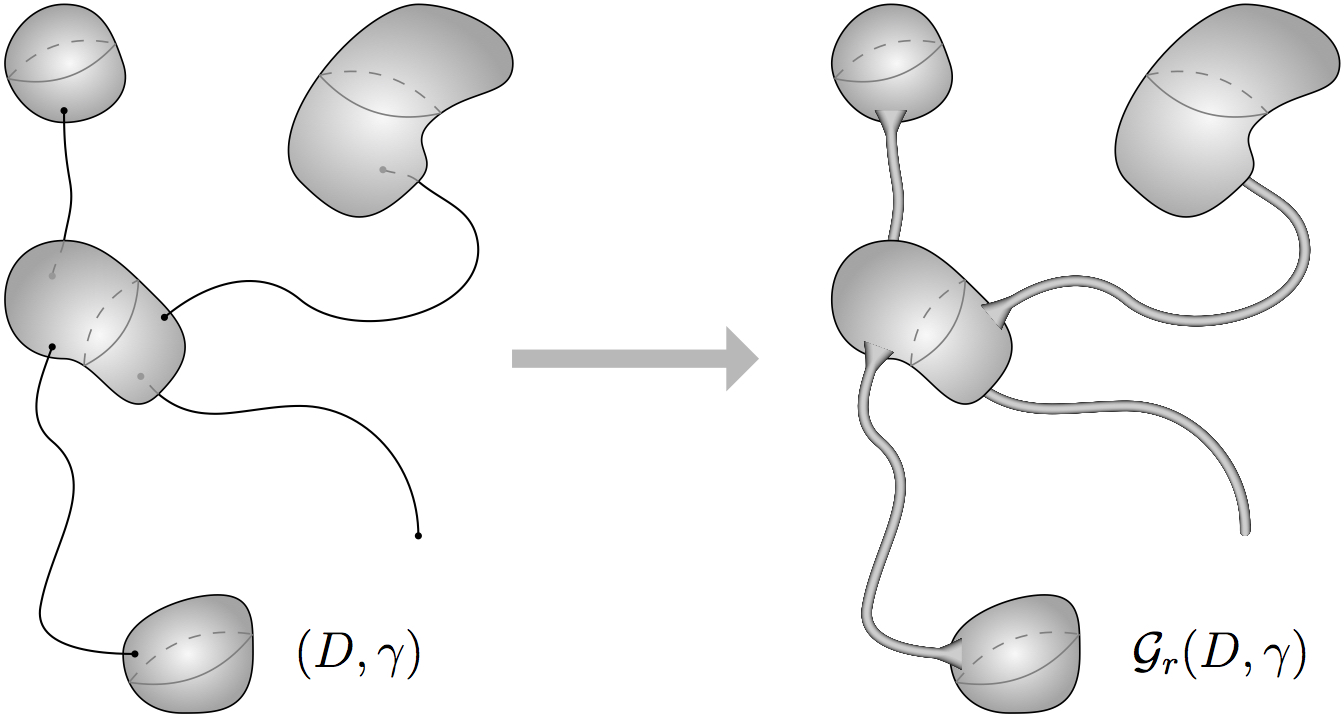}
\caption{An illustration of the gluing map.}
\end{figure}

To construct the gluing map, we start with the round model case.

\begin{proposition}[Gluing ball and cylinder]\label{gluing_sphere_cylinder}
There exists a smooth one-parameter family of smooth, 2-convex, rotationally symmetric, contractible domains $\{ C_\delta\subset \R^{n+1}\}_{\delta\in(0,1)}$, such that
\begin{enumerate}[\hspace{3mm}(1)]
\item $C_\delta \cap \{x_1 \leq 1-\delta\}= \{(x_1,\ldots,x_{n+1})\in\mathbb{R}^{n+1}\, | \sum_{i=1}^{n+1} x_i^2 \leq 1,\,x_1 \leq 1-\delta\}$,
\item $C_\delta \cap \{x_1 \geq 1+\delta\}= \{(x_1,\ldots,x_{n+1})\in\mathbb{R}^{n+1}\, | \sum_{i=2}^{n+1}x_i^2\leq\varrho(\delta)^2,\,x_1 \geq 1+\delta\}$,
\end{enumerate}
where $\varrho:(0,1)\rightarrow \mathbb{R}_+$ is smooth and increasing with $\lim_{\delta\to 0}\varrho(\delta)= 0$.

Moreover, if we express $\partial C_\delta$ as surface of revolution of a function $u_\delta(x)$, then we have
\begin{equation}\label{eq_sec_der_less_one}
u''_{\delta} <1 \qquad\textrm{ for all } \delta\geq 0.99.
\end{equation} 
\end{proposition}

\begin{proof}[{Proof of Proposition \ref{gluing_sphere_cylinder}}]
If $\partial C_\delta\subset \R^{n+1}$ is the surface of revolution of a smooth positive function $u_\delta(x)=u(x)$, the 2-convexity condition becomes
\begin{equation}\label{mean_convex_ineq}
P[u]=1+u'^2-uu''>0\, .
\end{equation}
We will construct a suitable function $u(x)$ step by step.

We start with $u(x)=\sqrt{1-x^2}$ for $x\leq 1-\delta$.

Next, we  transition  $u$ to a parabola around $t_{\ast}=1-\frac{10^9-1}{10^9}\delta$ according to Proposition \ref{Third_order_trim} (with $\eps=\frac{1}{2\cdot 10^9}\delta$) so that condition \eqref{mean_convex_ineq} is preserved, and such that, setting $\delta_1=u(t_{\ast})$,
\begin{equation}
u(t_{\ast}+t)=\delta_1-\frac{\sqrt{1-\delta_1^2}}{\delta_1}t-\frac{1}{2\delta_1^3}t^2
\end{equation}
in a right neighborhood of $t=0$.

Next, we will deform the coefficient of the quadratic term. To this end, let $a(t)=\tfrac{1}{\delta_1^3}\varphi(\tfrac{t}{\delta_2})$, where $\varphi(\cdot)$ is the transition function from Proposition \ref{good_transition}, and $\delta_2>0$ is a small parameter to be determined below. Note that $a \equiv -\frac{1}{2\delta_1^3}$ in a neighborhood of $t=0$, and $a \equiv +\frac{1}{4\delta_1^3}$ in a neighborhood of $t=\delta_2$. Furthermore, using the inequality $Q[a]\leq 0.95/\delta_1^3$ we see that
\begin{equation}
P\left[\delta_1-\tfrac{\sqrt{1-\delta_1^2}}{\delta_1}t+a(t)t^2\right]>0
\end{equation}
for all $t\in[0,\delta_2]$, provided we choose $\delta_2\ll \delta_1$.

We can thus continue with the parabola
\begin{equation}
p(t)=\delta_1-\frac{\sqrt{1-\delta_1^2}}{\delta_1}t+\frac{1}{4\delta_1^3}t^2\qquad (t\geq\delta_2).
\end{equation}
Observe that the minimum of $p(t)$ is at $t_{\min}=2\delta_1^2\sqrt{1-\delta_1^2}$, and that
\begin{equation}\label{eq_min_p}
p(t_{\min})=\delta_1^3>0.
\end{equation}
Moreover, the parabola $p(t)$ indeed satisfies condition \eqref{mean_convex_ineq}, since
\begin{equation}
P[p]=\frac{1}{8\delta_1^6}t^2-\frac{\sqrt{1-\delta_1^2}}{2\delta_1^4}t+\frac{1}{2\delta_1^2}=\frac{1}{2\delta_1^3}\,p(t) \geq \frac12.
\end{equation}
Finally, noticing that once we were able to achieve $u'=0$ with $u>0$, finding a canonical mean convex extension to the cylinder is easy by the gluing principle (c.f. Remark \ref{gluing_principle}), this finishes the construction of $u$.

Tracing trough the steps, it is clear that $u_\delta$ can be constructed depending smoothly on $\delta$. 
From \eqref{eq_min_p} we see that $\lim_{\delta\to 0}\varrho(\delta)=0$.
Finally, it is clear by construction that $u''_{\delta} <1$ for all $\delta\geq 0.99$.
\end{proof}

We will now reduce the general case to the rotationally symmetric one. For $\lambda=(\lambda_1,\ldots,\lambda_{n})\in \mathbb{R}^n$, letting $\underline{\lambda}=\frac{1}{n}\sum \lambda_i$, we consider the function
\begin{equation}
u^{\lambda}(x_1,\ldots,x_n):=\underline{\lambda}^{-1}\left(1-\sqrt{1-\underline{\lambda}\sum_{i=1}^{n}\lambda_ix_i^2}\right).
\end{equation}

\begin{proposition}[Deforming mixed curvatures]\label{mixed_curvature_deform}
There exists a constant $c=c(\mathbb{A})>0$, and a smooth $(n+1)$-parameter family of smooth 2-convex contractible domains $C_\delta^\lambda$ in $\mathbb{R}^{n+1}$, that have a boundary component in $\{\sum_{i=1}^{n}x_i^2\leq 2{\delta}\}$ satisfying
\begin{enumerate}[\hspace{3mm}(1)]
\item $\partial C_\delta^{\lambda} \cap \{\delta \leq \sum_{i=1}^nx_i^2 \leq 2\delta\}= \graph\left(u^{\lambda}\right)$,
\item $\partial C_\delta^{\lambda} \cap \{\sum_{i=1}^nx_i^2 \leq c\delta \}=\graph\left(u^{(\underline{\lambda},\ldots, \underline{\lambda})}\right)$.
\end{enumerate}
The family is defined for $\lambda=(\lambda_1,\ldots,\lambda_n)$ with $-C_A \leq \lambda_1\leq \lambda_2 \leq \ldots \leq \lambda_n \leq C_A$ and $\lambda_1+\lambda_2\geq \beta c_H$ and $\delta\in (0,\bar{\delta})$, where $\bar{\delta}=\bar{\delta}(\mathbb{A})>0$. 

\noindent Moreover, if $\lambda_1=\ldots =\lambda_n$, then $\partial C_\delta^{\lambda}= \graph\left(u^{\lambda}\right)\cap \{\sum_{i=1}^{n}x_i^2 < 2{\delta}\}$.
\end{proposition}

Before giving a proof, let us summarize some generalities about computing the curvature of graphs. Let $u:B^n_{2\delta}(0)\rightarrow \mathbb{R}$ be a smooth function. In the parametrization $(x_1,\ldots,x_n)\mapsto (x_1,\ldots,x_n,u(x_1,\ldots,x_n))$ for $\graph(u)$, the second fundamental form is given by   
\begin{equation}\label{sc_fun_form_graph}
h^{i}_j=\frac{1}{\sqrt{1+|Du|^2}}\left(1-\frac{1}{1+|D u|^2}\frac{\D u}{\D x_i}\frac{\D u}{\D x_j}\right)\frac{\D^2u}{\D x_i\D x_j}.
\end{equation}

Let $\textrm{Sym}_n$ denote the space of all real symmetric $n\times n$ matrices and let $S:\textrm{Sym}_n\rightarrow \mathbb{R}$ be the function which assigns to each symmetric matrix the sum of its smallest two eigenvalues. Since $S$ is uniformly continuous, for every $\kappa>0$ there exist some $\omega(\kappa)>0$ such that if $S(A)\geq\kappa$ then $S(A+E)\geq\kappa/2$ for every $E\in \mathrm{Sym}_n$ with $|E^i_j|<\omega$.

\begin{proof}[{Proof of Proposition \ref{mixed_curvature_deform}}] We make the ansatz
\begin{equation}
u(x)=\underline{\lambda}^{-1}\left(1-\sqrt{1-\underline{\lambda}\sum_{i=1}^{n}\left((1-\eta(r))\lambda_i+\eta(r)\underline{\lambda}\right)x_i^2}\right),
\end{equation}
where $r(x)=\sqrt{\sum_{i=1}^{n}x_i^2}$ and $\eta:\mathbb{R}\to [0,1]$ is a smooth function to be specified below.

By \eqref{sc_fun_form_graph}, we have that for $r<\delta<\bar{\delta}(\mathbb{A})$,
\begin{equation}
h^{i}_j=((1-\eta(r))\lambda_i+\eta(r)\underline{\lambda})\delta^{i}_j+E^i_j,
\end{equation}

where the error term can be estimated by
\begin{equation}
|E^{i}_j| \leq C_1(C_A)\left(r+|r\eta'(r)|+|r^2\eta''(r)|\right).
\end{equation} 

Let $\psi$ be the function from Proposition \ref{trasnsition_function_b} with $\eps=\omega(\beta c_H)/2C_1(C_A)$. Letting $\eta(r)=\psi(r/\delta)$, we see from the above that $S(h^{i}_j)>0$, possibly after reducing $\bar{\delta}$.

We can thus transition between the functions $u^{\lambda}$ and $u^{(\underline{\lambda},\ldots,\underline{\lambda})}$, as $r$ decreases from $\delta$ to $c\delta>0$, where $c=c(\mathbb{A})>0$ is chosen such that $\psi|_{[0,c]}=1$. This implies the assertion. 
\end{proof}

\begin{proposition}[Second order approximation]\label{approx_parab}
There exist a constant $\bar{\delta}=\bar{\delta}(\mathbb{A})>0$ and a smooth, rigid motion equivariant map
\begin{equation*}
\mathcal{D}_{\mathbb{A}}^{\textrm{pt}}\times (0,\bar{\delta})\rightarrow \mathcal{D},\quad ((K,p),\delta)\mapsto \mathcal{P}_\delta(K,p),
\end{equation*}
 such that setting $K'=\mathcal{P}_\delta(K,p)$ it holds that:
\begin{enumerate}[\hspace{3mm}(1)]
\item $K\cap B_{\delta}(p)$ and $K'\cap B_{\delta}(p)$ are diffeomorphic to a half-ball.
\item $K\setminus B_{\delta}(p)=K' \setminus B_{\delta}(p)$.
\item If $\lambda_1\leq \ldots  \leq \lambda_n$ are the principal curvatures of $K$ at $p$ then setting $\lambda=(\lambda_1,\ldots,\lambda_n)$, in some orthonormal basis, 
\begin{equation*}
(\partial K'-p)\cap B_{\delta/2}(p)=\graph\left(u^{\lambda}\right)\cap B_{\delta/2}(0).   
\end{equation*}
\end{enumerate}
Moreover, if $K\cap B_{\delta}(p)$ is round, then $K'=K$.
\end{proposition}
\begin{proof}
The $\alpha$-noncollapsing condition, together with the bounds for $\abs{A}$ and $\abs{\nabla A}$ from \eqref{eq_bound_contr_dom} implies that, after rotation, in a neighborhood of the origin of radius $\bar{\delta}=\bar{\delta}(\mathbb{A})>0$, one can write 
\begin{equation}
\partial K-p=\graph(u(x_1,\ldots,x_n)),
\end{equation}
such that the error term $E:=u-u^{\lambda}$ satisfies
\begin{equation}\label{eq_quad_err_est}
\qquad\qquad\qquad |D^{(i)}E| \leq C\delta^{3-i}\qquad\qquad (i=0,1,2)
\end{equation} for $r(x)=\sqrt{\sum_{i=1}^{n}x_i^2}<\delta<\bar{\delta}(\mathbb{A})$, where $C=C(\mathbb{A})<\infty$. 

Moreover, possibly after decreasing $\bar{\delta}$, computing $h^i_j$ of $\graph(u^\lambda)$ as in \eqref{sc_fun_form_graph} we see that  
\begin{equation}\label{definitely_positive_graph}
S(h^{i}_j)\geq \frac{\beta c_H}{2}
\end{equation}
in that neighborhood.

Given $\delta<\bar{\delta}$, we let $\eta(r)=\chi_{\delta/2}(r-\delta/2)$ and make the ansatz
\begin{equation}
v=u^{\lambda}+\eta E\, .
\end{equation}
Since $|D^{(i)}\eta | \leq C \delta^{-i}$ for some $C<\infty$, together with \eqref{eq_quad_err_est} we get
\begin{equation}\label{small interp_remainder}
\qquad\qquad\qquad |D^{(i)}(\eta E)| \leq C\delta^{3-i} \qquad\qquad (i=0,1,2)\, .
\end{equation}  
Denoting by $\bar{h}^i_j$ the second fundamental form of $\graph(v)$, it follows from \eqref{small interp_remainder} that, possibly after decreasing $\bar{\delta}$ further, $|h^{i}_j-\bar{h}^i_j|<\omega(\beta c_H/2)$. Thus $S(\bar{h}^{i}_j)>\beta c_H/4$ and the assertion follows.
\end{proof}

\begin{proposition}[Bending curves]\label{bending_tubes}
There exists a constant $\bar{\delta}=\bar{\delta}(b)>0$ and a smooth, rigid motion equivariant map
\begin{equation}
\mathcal{C}^{\textrm{pt}}_{b}\times (0,\bar{\delta}) \rightarrow \mathcal{C}^{\textrm{pt}}_{10^{-6}b},\quad ((\gamma,p),\delta)\mapsto \mathcal{B}_\delta (\gamma,p),
\end{equation}
such that setting $\mu=\mathcal{B}_\delta(\gamma,p)$ we have:
\begin{enumerate}[\hspace{3mm}(1)]
\item $\mu\setminus B_{\delta}(p)=\gamma\setminus B_{\delta}(p)$,
\item $\mu\cap B_{\delta/2}(p)$ is a straight ray starting at $p$ with $\partial_s\mu(p)=\partial_s\gamma(p)$.
\item $d_H(\mu,\gamma)<10^{-2}\delta$, where $d_H$ denotes the Hausdorff distance.
\end{enumerate} 
Moreover, if $\gamma\cap B_{\delta}(p)$ is straight, then $\mu=\gamma$.
\end{proposition}

\begin{proof}
Let $\ell(t):=p+tv$, where $v=\partial_s\gamma(p)$.
We would like to interpolate between $\ell$ and $\gamma$.

Let $\eta(t)=\chi_{\delta/2}(t-\delta/2)$ where $\chi_{\delta/2}$ is the standard cutoff function from \eqref{eq_stepfn}. Parametrize $\gamma$ by arclength and let
\begin{equation}
\mu =\eta \gamma +(1-\eta)\ell \, .
\end{equation}
Since $\abs{(\ell(t)-\gamma(t))^{(j)}}\leq 10 b^{-1}t^{2-j}$ ($j=0,1,2$) for $t\in [0,\delta]$ and $\delta<\bar{\delta}=\bar{\delta}(b)$ small enough, we get the estimates $\abs{\mu'(t)-v} \leq 10^{3}\delta b^{-1}$ and $\abs{\mu''(t)}\leq 10^{3}b^{-1}$. In particular, $\mu$ has curvature bounded by $10^{6}b^{-1}$. Decreasing $\bar{\delta}=\bar{\delta}(b)$ even more if needed we can ensure that (3) holds and that the curves do not reenter $B_\delta(p)$, by $b$-controlledness. This proves the assertion.
\end{proof}

\begin{proof}[Proof of Theorem \ref{thm_attachstrings}]
Set $\delta(r):=\frac{2\sqrt{2}n}{c c_H}\sqrt{\varrho^{-1}(C_A r)}$, where $c=c(\mathbb{A})$ is the constant from Proposition \ref{mixed_curvature_deform},
$c_H$ and $C_A$ are the curvature bound from \eqref{eq_bound_contr_dom}, and $\varrho$ is the function from Proposition \ref{gluing_sphere_cylinder}. 

Choose $\bar{r}(\mathbb{A},b)>0$ smaller than $10^{-6}\zeta(\alpha,\beta)b$, where $\zeta(\alpha,\beta)$ is from Proposition \ref{mean_convex_capped_off}, small enough such that $\varrho^{-1}(C_A \bar{r})<1$ and $\delta(\bar{r})<b/100$, and small enough such that we can apply Propositions
\ref{bending_tubes},
\ref{approx_parab},
\ref{mixed_curvature_deform}, and
\ref{gluing_sphere_cylinder} for $\delta(r)$ with $r\in (0,\bar{r})$.

Given $(D,\gamma)\in \mathcal{X}_{\mathbb{A},b}$ and $r\in (0,\bar{r})$, we first bend the ends of $\gamma$ that don't hit $\partial D$ straight at scale $\delta(r)$ using Proposition \ref{bending_tubes}. Next, by Proposition \ref{approx_parab}, for every $p\in \partial \gamma\cap \partial D$ we can deform the domain inside $B_{\delta(r)}(p)-B_{\delta(r)/2}(p)$ to its second order approximation in $B_{\delta(r)/2}(p)$. Next, by Proposition \ref{mixed_curvature_deform} we can transition to $\graph\left(u^{(\underline{\lambda},\ldots,\underline{\lambda})}\right)\cap B_{c\delta(r)/2}$. Finally, we can apply Proposition \ref{gluing_sphere_cylinder} with
$\delta=\varrho^{-1}(\underline{\lambda}r)$ and rescale (note that $\frac{\sqrt{2\delta}}{\underline{\lambda}}\leq \frac{c\delta(r)}{2}$ by our choice of constants) to glue this to the $r$-tubular neighborhood of the straight end of the curve. 
By Proposition \ref{mean_convex_capped_off} the $r$-tubular neighborhoods and the capped off tubes are indeed 2-convex, and this finishes the construction of $\mathcal{G}_r(D,\gamma)\in\mathcal{D}$.

The fact that the above construction indeed provides a smooth rigid motion invariant map as stated is clear since the choices that were made were done via specific functions.  The fact that $\delta(r)$ is increasing and $\lim_{r\rightarrow 0}\delta(r)=0$ follows from the corresponding fact about the function $\varrho$ of Proposition \ref{gluing_sphere_cylinder}. Conditions (1)--(4) are clear from the construction as well. 

Finally, the more precise description in the case that $D\cap B_{\delta(r)}(p)$ is round, follows from the `moreover'-parts of Propositions \ref{bending_tubes}, \ref{approx_parab}, and \ref{mixed_curvature_deform}.
\end{proof}

\section{Marble trees}\label{sec_marble}

The goal of this section is to show that every marble tree can be deformed via 2-convex domains to a round ball.

\begin{definition}[Marble tree]\label{def_marbletree}
A \emph{marble tree} with string radius $r_s$ and marble radius $r_m$ is a domain of the form $T:=\mathcal{G}_{r_s}(D,\gamma)$ such that
\begin{enumerate}
\item $D=\bigcup_i \bar{B}_{r_m}(p_i)$ is a union of finitely many balls of radius $r_m$.
\item The curve $\gamma$ is such that:
\begin{itemize}
\item there are no loose ends, i.e. $\partial\gamma\setminus\partial D=\emptyset$,
\item $\gamma\cap \bar{B}_{3r_m}(p_i)$ is a union of straight rays for all $i$,
\item and $D\cup \gamma$ is simply connected.
\end{itemize}
\end{enumerate}
\end{definition}

\begin{theorem}[Marble tree isotopy]\label{thm_marble_tree}
Every marble tree is isotopic through 2-convex domains to a round ball.
\end{theorem}

We will prove Theorem \ref{thm_marble_tree} by induction on the number of marbles. If there is just one marble, then our marble tree is just a round ball and there is no need to deform it further.

Assume now $T:=\mathcal{G}_{r_s}(D,\gamma)$ is a marble tree with at least two marbles. The induction step essentially amounts  to selecting a leaf, i.e. a ball with just one string attached, deforming this leaf to a standard cap, and contracting the tentacle. To actually implement the induction step, let us start with some basic observations and reductions:

By the `moreover'-part of Theorem \ref{thm_attachstrings} the gluing is described by the explicit round model case from Proposition \ref{gluing_sphere_cylinder}. In particular, $T:=\mathcal{G}_{r_s}(D,\gamma)$ is independent of the precise values of the control parameters $(\mathbb{A},b)$, and we can thus adjust the control parameters freely as needed in the following argument (e.g. we can relax the control parameters so that we can shrink the marbles). We will suppress the control parameters in the following, and we will simply speak about controlled configurations when we mean configurations that are controlled for some parameters $(\mathbb{A},b)$.

Let $\mathcal{X}$ be the space of controlled configurations $(D,\gamma)$ that satisfy properties (1) and (2) from Definition \ref{def_marbletree} for some $r_m$.

\begin{lemma}[Rearrangements]\label{tentacle_rearange_lem}
Let $(\bigcup_i \bar{B}_{r_m}(p_i),\bigcup_j \gamma_j)\in \mathcal{X}$, and suppose (after relabeling) that $\gamma_1$ meets $\partial B_{r_m}(p_1)$. Then there exists a smooth path in $\mathcal{X}$ connecting it to a configuration $(\bigcup_i \bar{B}_{\tilde{r}_m}(p_i),\bigcup_j \tilde{\gamma}_j)\in \mathcal{X}$ such that $\tilde{\gamma}_1$ is the only curve meeting $B_{\tilde{r}_m}(p_1)$ in its hemisphere. Moreover, the path can be choosen to be trivial outside $\bigcup_i \bar{B}_{r_m}(p_i)$.
\end{lemma}

\begin{figure}[H]
\includegraphics[width=13cm]{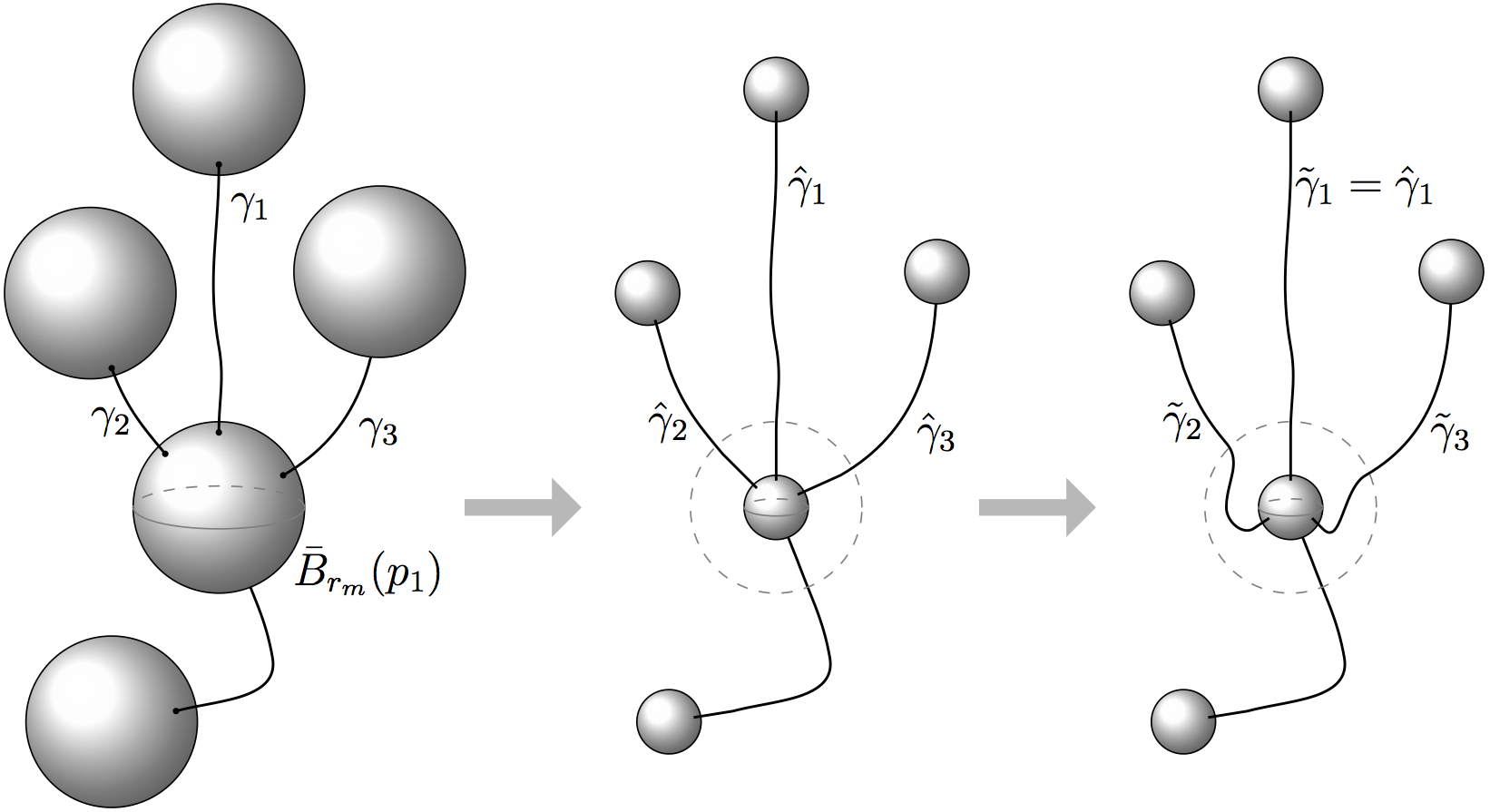}
\caption{The idea of the proof of the Rearrangements Lemma is to first shrink the marble radius and then use the space gained this way to move the curves that meet $\bar{B}_{r_m}(p_1)$ in the same hemisphere as $\gamma_1$.}
\end{figure}

\begin{proof}
This is done by inducting on the number of curves meeting $\bar{B}_{r_m}(p_1)$ in the hemisphere of $\gamma_1$.  If there are no other curves, we are done. Otherwise, assume without loss of generality that $\gamma_2$ meets $\bar{B}_{r_m}(p_1)$ in the hemisphere. We now construct a smooth path in $\mathcal{X}$ as follows:

We first shrink the marble radius by a factor $100$ and extend all curves by straight lines. This yields a configuration $(\bigcup_i \bar{B}_{\hat{r}_m}(p_i),\bigcup_j \hat{\gamma}_j)\in \mathcal{X}$ such that $\hat{\gamma}_j\cap \bar{B}_{100\hat{r}_m}(p_i)$ is a straight ray (possibly empty) for all $i,j$.

To simplify notation assume without loss of generality that $p_1=0$ and $\hat{r}_m=1$. Let $\eta(r)=\chi_{10}(20-r)$. Choose a smooth curve $\mu(t)\in \partial B_{1}(0)$ from $\mu(0)=\hat{\gamma}_2\cap \bar{B}_1(0)$ to a point $\mu(1)$ on the other hemisphere, such that $\mu(t)\cap \gamma=\emptyset$ for all $t>0$. Consider the deformation
\begin{equation}
\gamma^t_2(r)=r\mu(\eta(r)t) \qquad r\in [1,100].
\end{equation}

Then $\gamma^t_2(r)=r \mu (t)$ for $r\in [1,3]$ for every $t\in [0,1]$, which is a straight line as required by condition (2) from Definition \ref{def_marbletree}, and $\gamma^t_2(r)=\gamma^0_2(r)$ for $r\geq 20$. Thus $\gamma^t_2$ gives a smooth path in $\mathcal{X}$ starting from $\gamma^0_2(r)=r\mu(0)=\hat{\gamma}_2(r)$. This modification decreases the number of meetings in the hemisphere defined by $\hat{\gamma}_1$, and we are thus done by the induction hypothesis.
\end{proof}

Reducing $r_s$ and applying $\mathcal{G}_{r_s}$ to the path of configurations from Lemma \ref{tentacle_rearange_lem} we can move strings of a marble tree to the other hemisphere. Another basic reduction is to further decrease the marble radius and the string radius, to ensure that suitable neighborhoods don't contain any other marbles or strings and to ensure that tubular neighborhoods of curves are mean convex even at the marble scale.

By the above reductions and rescaling, to prove Theorem  \ref{thm_marble_tree} it is thus enough to prove the following proposition.

\begin{proposition}[Marble reduction]\label{leaf_contract}
Let $(\bigcup_{i\geq 1} \bar{B}_1(p_i),\bigcup_{j\geq 1} \gamma_j)\in \mathcal{X}$ be a controlled configuration, and suppose (after relabeling) that $\gamma_1$ connects $\partial \bar{B}_1(p_1)$ to $\partial \bar{B}_1(p_2)$, that no other curves meet $\partial \bar{B}_1(p_1)$, and that the $\sqrt{2}$-neighborhood of $\gamma_1$ does not intersect any other balls or curves. Also assume that $b$ is bigger than $\zeta^{-1}(\alpha,\beta)$ from Proposition \ref{mean_convex_capped_off}.

Then for $r_s$ small enough, the marble tree $\mathcal{G}_{r_s}(\bigcup_{i\geq 1} \bar{B}_{1}(p_i),\bigcup_{j\geq 1} \gamma_j)$ is isotopic via 2-convex domains to $\mathcal{G}_{r_s}(\bigcup_{i\geq 2} \bar{B}_{1}(p_i),\bigcup_{j\geq 2} \gamma_j)$.
\end{proposition}

\begin{proof}
We prove this proposition in five steps which are illustrated in Figure \ref{fig.5steps} below. Recall that Proposition \ref{gluing_sphere_cylinder} gives an explicit description of how the connected sum looks like. This  allows us to make all the calculations in terms of the standard model $C_{\delta}$.

\emph{Step 1:} Since $C_{\delta}$ depends smoothly on $\delta\in (0,1)$ and since the control parameter of the curve satisfies $b>\zeta^{-1}(\alpha,\beta)$ by assumption, we can increase the radius of the neck around $\gamma_1$ from $r_s$ to $r=\varrho(0.99)$.

\emph{Step 2:} We deform this configuration to a standard cap as follows. Let $q_1$ be the point where $\gamma_1$ meets $\partial B_1(p_1)$. Without loss of generality we assume $p_1=0$ and $q_1=(1,0,\ldots,0)$. Let $\tilde{\gamma}_1=\gamma_1\cup \ell$ where $\ell$ is the straight line from $-q_1$ to $q_1$.
Consider a standard cap $K$ of neck-radius $r=\varrho(0.99)$, tangent to $C_{0.99}$ at $-q_1$. Note that $K\subset C_{0.99}$ by definition of the standard cap. Expressing both domains as graphs of rotation, by \eqref{eq_sec_der_less_one} and Proposition \ref{linear_interpolation} we can find an isotopy from $C_{0.99}$ to $K$ via two-convex domains that stays inside $C_{0.99}$ and is trivial outside $C_{0.99}\cap B_2(0)$.

\emph{Step 3:} We contract the tentacle as follows. We have created a left capped-off tube
$CN^-_r(\tilde{\gamma}_1,-q_1)$. Now, let $\tilde{\gamma}_1:[0,L]\to \mathbb{R}^{n+1}$ be the parametrization by arclength starting at $\tilde{\gamma}_1(0)=-q_1$. The tentacle can then be contracted via the isotopy $\{CN^-_r(\tilde{\gamma}_1,\tilde{\gamma}_1(t))\}_{t\in [0,L-2]}$.

After rigid motion and ignoring the strings on the other hemisphere, we are thus left with the domain
\begin{equation}
D:=\left(C_{0.99}\cap \{x_1 \leq 2\}\right)\cup \left(CN^+_{r}(\gamma,(3,0,\ldots,0))\cap \{x_1 \geq 2\}\right),
\end{equation}
where $\gamma(s)=(s,0,\ldots,0)$ with $s\in [1,3]$. The final task is to find a 2-convex isotopy from $D$ to $\bar{B}_1(0)$ that is trivial in the left half-space.

\emph{Step 4:} We first want to deform $D$ to a convex domain $\tilde{D}$. In order to do so, let $u:[-1,3]\rightarrow [0,1]$ be the function whose surface of revolution is $D$. By \eqref{eq_sec_der_less_one} we have $u''<1$. For $\eps>0$ sufficiently small there exists a smooth function $v:[-1,3]\rightarrow [0,1]$ with
\begin{enumerate}[\hspace{3mm}(1)]
\item $v=\sqrt{1-x^2}$ on $[-1,\eps]$,
\item $v$ is concave,
\item $v\geq u$,
\item $\lim_{x\rightarrow 3} v(x)=0$ and the surface of revolution of $v$ is smooth.
\end{enumerate}
By Proposition \ref{linear_interpolation} the domain $D$ can be isotoped through 2-convex domains to the convex domain $\tilde{D}$ given by $v$. 

\emph{Step 5:} As $\tilde{D}$ is convex, $t\bar{B}_1(0)+(1-t)\tilde{D}$ is the desired isotopy deforming $\tilde{D}$ to $\bar{B}_1(0)$ while keeping $\{x_1 \leq \eps\}$ fixed.  
\end{proof}

\begin{figure}[H]
\includegraphics[width=13.5cm]{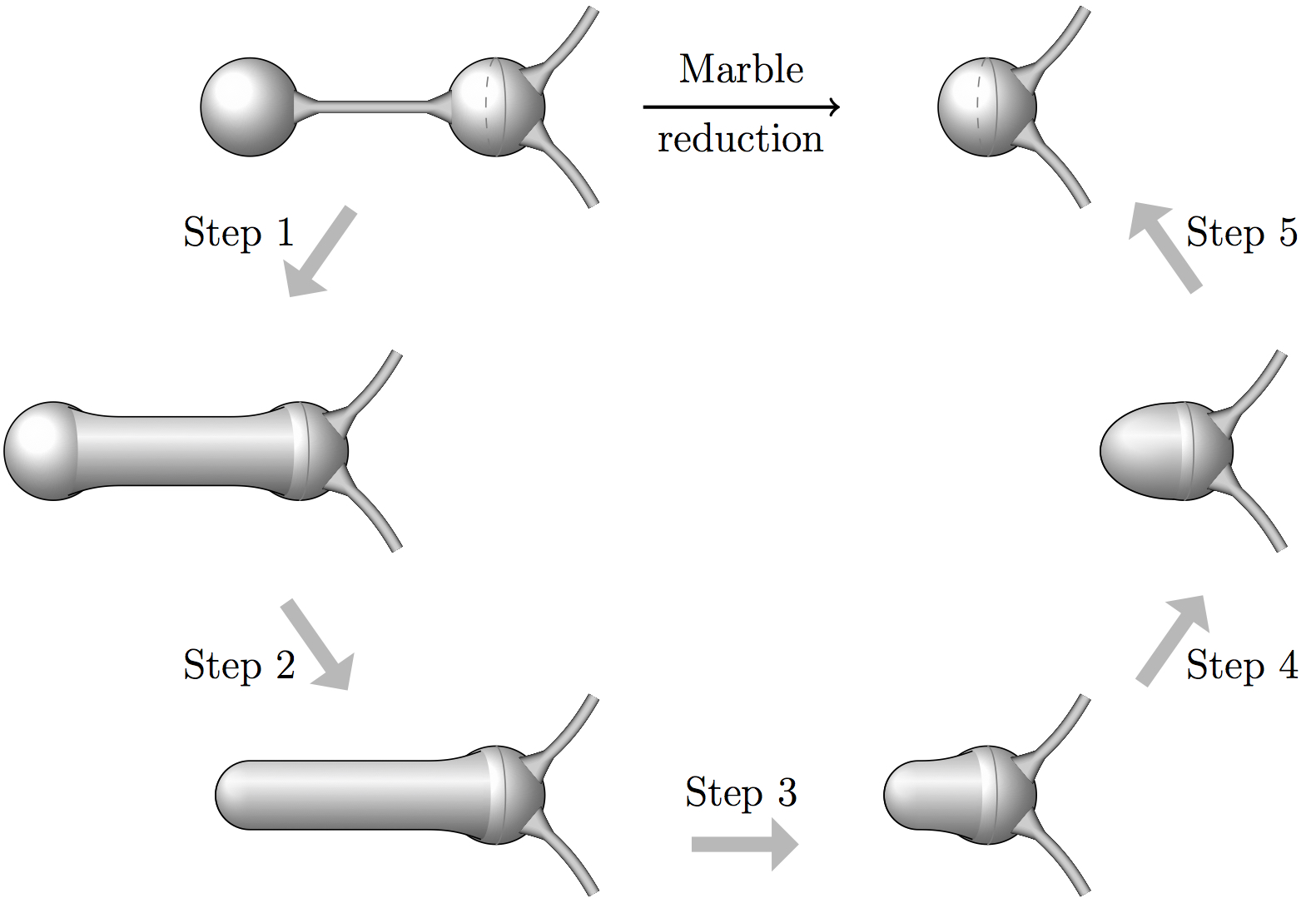}
\caption{An illustration of he five steps in which we construct the isotopy that reduces the number of marbles.}\label{fig.5steps}
\end{figure}

\section{Necks and surgery}\label{sec_necksurgery}

The goal of this section is to analyze the transition from the pre-surgery domain to the post-surgery domain.
We start by recalling some definitions from \cite{HK_surgery}.

\begin{definition}[{$\delta$-neck \cite[Def 2.3]{HK_surgery}}]
We say that $K\subset \R^{n+1}$ has a \emph{$\delta$-neck} with center $p$ and radius $r$, if
$r^{-1}\cdot( K-p)$ is $\delta$-close in $C^{\lfloor 1/\delta\rfloor}$ in $B_{1/\delta}(0)$ to a solid round cylinder $\bar{D}^{n}\times \R$ with radius $1$.
\end{definition}

\begin{definition}[{Replacing a $\delta$-neck by standard caps \cite[Def 2.4]{HK_surgery}}]\label{def_surgery}
We say that a $\delta$-neck with center $p$ and radius $r$ is \emph{replaced by a pair of standard caps} with cap separation parameter $\Gamma$,
if the pre-surgery domain $K^-$ is replaced by a post-surgery domain $K^\sharp\subset K^-$ such that:
\begin{enumerate}[\hspace{3mm}(1)]
\item the modification takes places inside a ball $B=B_{5\Gamma r}(p)$.
 \item there are bounds for the second fundamental form and its derivatives:
$$\sup_{\D K^\sharp\cap B}\abs{\nabla^\ell A}\leq C_\ell r^{-1-\ell}\qquad (\ell=0,1,2,\ldots,20).\footnote{In the case of mean curvature flow with surgery one considers the final time slice of a strong $\delta$-neck and the derivative bounds then hold for all $\ell\geq 0$.}$$
 \item for every point $p_\sharp\in \partial K^\sharp\cap B$ with $\lambda_1(p_\sharp)< 0$, there is a point
 $p_{-}\in\partial K^{-}\cap B$ with $\frac{\lambda_1}{H}(p_{-})\leq\frac{\lambda_1}{H}(p_{\sharp})$.
 \item the domain $r^{-1}\cdot(K^\sharp-p)$ is $\delta'(\delta)$-close in $C^{\lfloor 1/\delta'(\delta)\rfloor}$ in $B_{10\Gamma}(0)$ to a pair of opposing standard caps (see Definiton \ref{def_stdcap}),
that are at distance $\Gamma$ from the origin, where $\delta'(\delta)\to 0$ as $\delta\to 0$.
\end{enumerate}
\end{definition}

\begin{definition}[Points modified by surgery]\label{def_modifiedpoints}
We say that an open set $U$ contains \emph{points modified by surgery} if $(K^-\setminus K^\sharp)\cap U\neq \emptyset$.
\end{definition}

\begin{lemma}[Almost straight line]\label{lemma_almost_straight}
There exists a function $\xi(\delta')$ with $\xi(\delta')\to \infty$ as $\delta'\to 0$ with the following significance.
If $K$ is $\delta'$-close in $C^{\lfloor 1/\delta'\rfloor}$ in $B_{10\Gamma}(0)$ to a pair of opposing standard caps that are at distance $\Gamma$ from the origin, then there exists an $\xi(\delta')$-controlled curve inside the $\xi(\delta')^{-1}$-neighborhood of the $x$-axis that
connects the two components of $K$ and meets $\partial K$ orthogonally.
\end{lemma}

\begin{proof}
This follows from a standard perturbation argument.
\end{proof}

\begin{proposition}[Deforming neck to caps connected by a string]\label{lemma_neckconnectedsum}
There exists a constant $\bar{\delta}>0$ with the following significance.
Let $\delta\leq\bar{\delta}$, assume $K^\sharp$ is obtained from $K^-$ by replacing a $\delta$-neck  with center $0$ and radius $1$ by a pair of standard caps, and let $\gamma$ be a $\xi(\delta'(\delta))$-controlled curve connecting the caps as in Lemma \ref{lemma_almost_straight}.
Then, for $r_s$ small enough, there exists an isotopy between  ${\mathcal G}_{r_s}(K^\sharp,\gamma)$ and $K^-$, that preserves two-convexity and is trivial outside $B_{6\Gamma}(0)$.
\end{proposition}

\begin{proof}
$\partial K^{\sharp}$ can be expressed as a graph in $B_{6\Gamma}(0)$ with small $C^{20}$-norm over a pair of opposing standard caps of distance $\Gamma$ from the origin. For $\delta$ small enough, we can thus find a two-convex isotopy $\{K^\sharp_t\}_{t\in[0,1]}$ 
that is trivial outside of $B_{6\Gamma}(0)$, starting at $K^{\sharp}_0=K^{\sharp}$, such that $K_1^{\sharp}\cap B_{5\Gamma}(0)$ is a pair of opposing standard caps of distance $\Gamma$ from the origin. The isotopy can be constructed such that it remains $\tilde{\delta}(\delta)$-close in $B_{10\Gamma}(0)$ to the pair of opposing standard caps, for some function $\tilde{\delta}(\delta)$ with $\tilde{\delta}(\delta)\to 0$ as $\delta\to 0$. Similarly, there is a family of curves $\{\gamma_t\}_{t\in[0,1]}$ as in the conclusion of Lemma \ref{lemma_almost_straight} connecting the $\tilde{\delta}(\delta)$-caps of $K^{\sharp}_t$, starting at $\gamma_0=\gamma$, such that $\gamma_1$ is a straight line connecting the tips of $K^{\sharp}_1$. For $\delta$ and $r_s$ small enough, $\{\mathcal{G}_{r_s}(K^{\sharp}_t,\gamma_t)\}_{t\in[0,1]}$ forms the first step of the asserted isotopy.

Taking into account property (a) of the standard cap (see Definition \ref{def_stdcap}) we can now increase the neck radius from $r_s$ to $\varrho(0.99)$, where $\varrho$ is the function from Proposition \ref{gluing_sphere_cylinder}. The function $u$ describing the resulting domain in $B_{5\Gamma}(0)$ and the constant function $v:=1$ satisfy the assumptions of Proposition \ref{linear_interpolation}, so we can deform to a cylinder in $B_{5\Gamma}(0)$. Finally, similarly as in the first step in this proof, we can deform this to our $\delta$-neck $K^-$.   
\end{proof}

\begin{figure}[H]
\includegraphics[width=10cm]{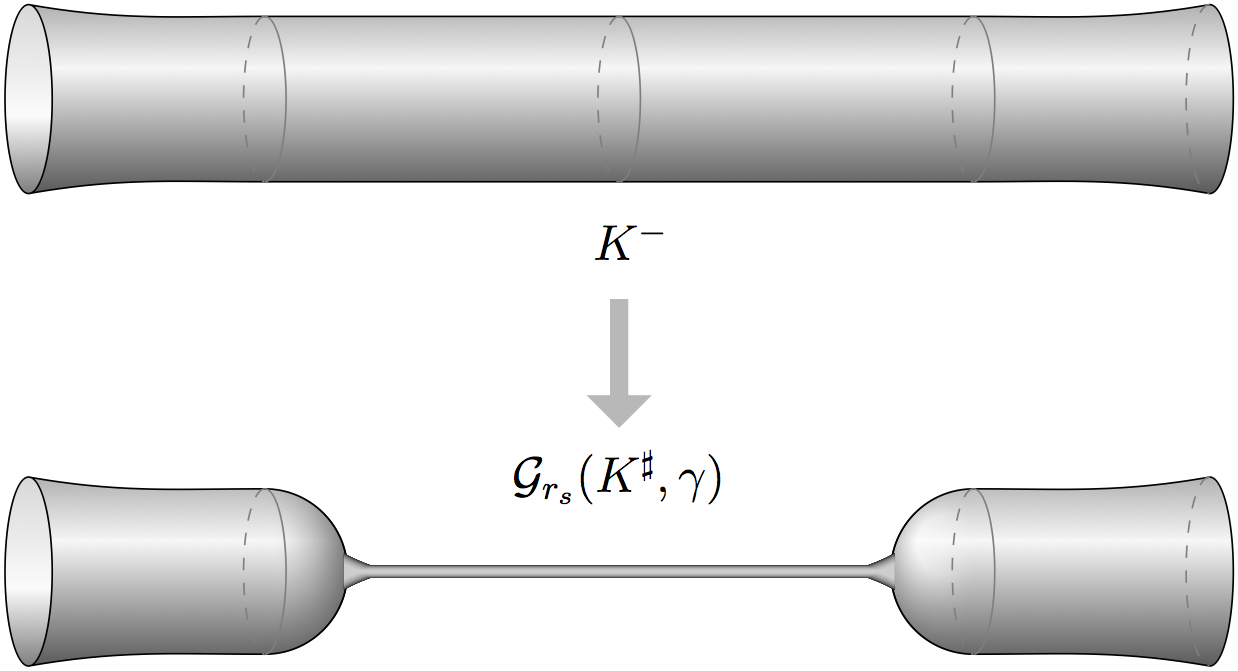}
\caption{Proposition \ref{lemma_neckconnectedsum} constructs an isotopy between the pre-surgery neck and a connected sum of the two post-surgery caps along an almost straight line.}
\end{figure}

\section{Isotopies for discarded components}\label{sec_discarded}

The goal of this section is to construct certain isotopies that will be used later in the proof of the main theorem to deal with the discarded components. We start with the following trivial observation.

\begin{proposition}[Convex domain]\label{lemma_convexcontrolled}
If $K\subset \mathbb{R}^{n+1}$ is a smooth compact convex domain, then there exists a monotone convex isotopy $\{K_t\}_{t\in [0,1]}$ that is trivial outside $K$, starting at $K_0=K$, such that $K_1$ is a round ball.
\end{proposition}
\begin{proof}
Choose a round ball $\bar{B}\subset K_0$. Then $K_t:=t \bar{B}+(1-t)K_0$ does the job.
\end{proof}

\begin{definition}[$(C,\eps)$-cap]\label{def_C_cap}
A $(C,\eps)$-cap is a  strictly convex noncompact domain $K\subset\mathbb{R}^{n+1}$ such that every point outside some compact subset of size $C$ is the center of an $\eps$-neck of radius 1.
\end{definition}

\begin{definition}[capped $\eps$-tube]\label{def_capped_off_chain}
A capped $\eps$-tube is a 2-convex compact domain $K\subset\mathbb{R}^{n+1}$ diffeomorphic to a ball, together with a controlled connected curve $\gamma\subset K$ with endpoints on $\partial K$ such that:
\begin{enumerate}
\item If $\bar{p}_\pm$ denote the endpoints of $\gamma$ then $K\cap B_{2CH^{-1}(\bar{p}_\pm)}(\bar{p}_\pm)$ is $\eps$-close (after rescaling to unit size) to either (a)  a $(C,\eps)$-cap (see Definition \ref{def_C_cap}) or (b) a standard-cap (see Definition \ref{def_stdcap}).
\item Every interior point $p\in \gamma$ with $d(p,\bar{p}_+)\geq CH^{-1}(\bar{p}_+)$ and $d(p,\bar{p}_-)\geq CH^{-1}(\bar{p}_-)$  is the center of an $\eps$-neck with axis given by $\partial_s\gamma(p)$. Moreover, if $r$ denotes the radius of the $\eps$-neck with center $p$, then $\gamma$ is  $\eps^{-2}r$-controlled in $B_{\eps^{-1}r}(p)$.
\end{enumerate}
\end{definition}

\begin{proposition}[Isotopy for capped $\eps$-tube]\label{lemma_neckwithstdcaps}
For $\eps$ small enough, every capped $\eps$-tube is isotopic via two-convex domains to a marble tree. Moreover, there exists a finite collection $\mathcal{I}$ of $\eps$-neck points with separation at least $100\max\{\eps^{-1},\Gamma\}\max\{H^{-1}(p),H^{-1}(q)\}$ for every pair $p,q\in \mathcal{I}$, such that the isotopy is monotone outside $\bigcup_{p\in \mathcal{I}}B_{6\Gamma H^{-1}(p)}(p)$.
\end{proposition}

\begin{figure}[H]
\includegraphics[width=12cm]{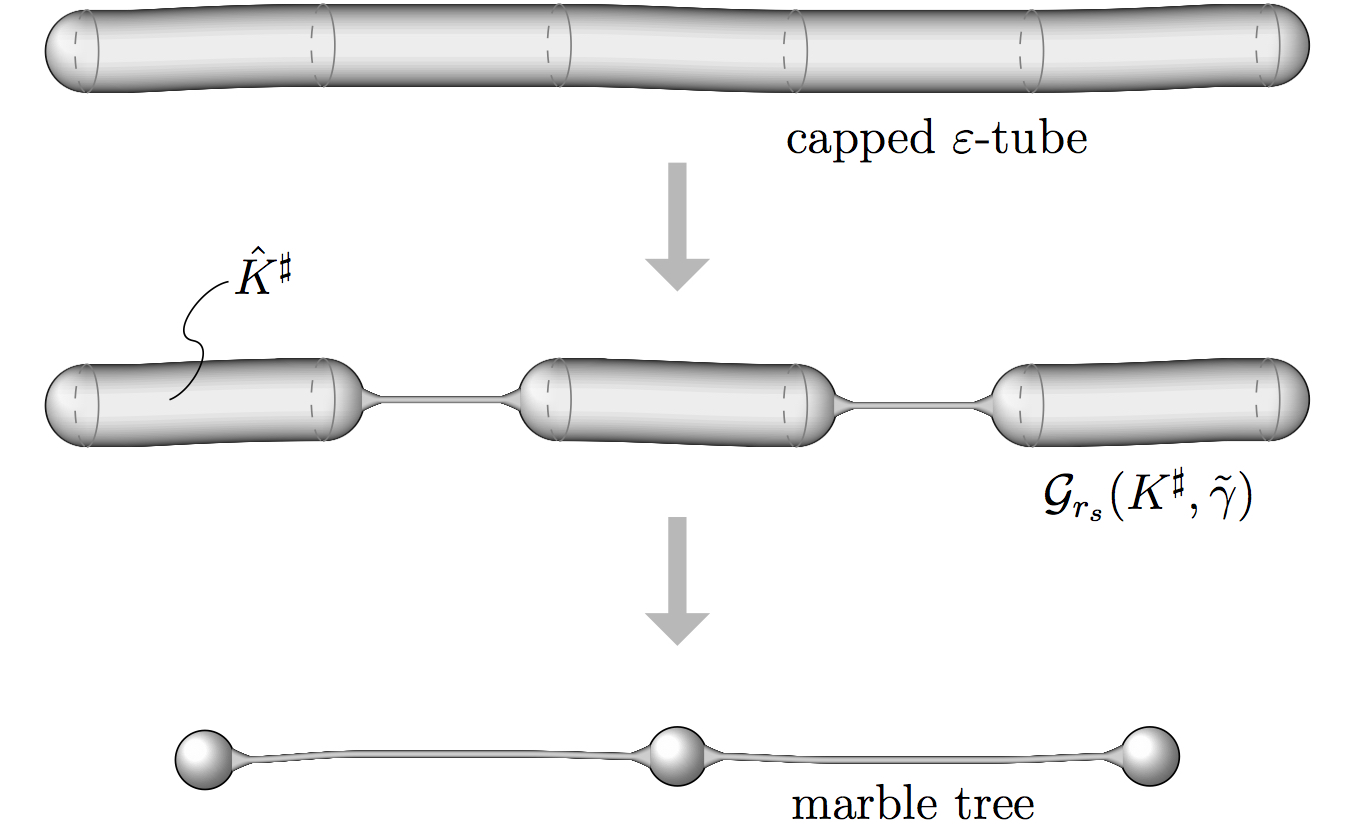}
\caption{Proposition \ref{lemma_neckwithstdcaps} is essentially proved in two steps: First, we use several surgeries and construct the corresponding isotopy as in Proposition \ref{lemma_neckconnectedsum} above. Then, we deform each connected component $\hat{K}^\sharp$ of $K^\sharp$ to a marble and extend the strings connecting them.}
\end{figure}

\begin{proof}
In the following we assume that $\eps$ and $r_s$ are small enough.

Let $p_\pm$ be $\eps$-neck points that are as close as possible to $\bar{p}_\pm$, respectively.
Let $\mathcal{I}\subset \gamma$ be a maximal collection of $\eps$-neck points with $p_\pm\in \mathcal{I}$ such that for any pair $p,q\in \mathcal{I}$ the separation between the points is at least $100\max\{\eps^{-1},\Gamma\}\max\{H^{-1}(p),H^{-1}(q)\}$.

For each $p\in\mathcal{I}$ we replace the $\eps$-neck with center $p$ by a pair of opposing standard caps as in Definition \ref{def_surgery}, which is possible by \cite[Prop. 3.10]{HK_surgery}. Denote the post-surgery domain by $K^\sharp$. Let $\tilde{\gamma}$ be the disjoint union of almost straight curves connecting the opposing standard caps as in Proposition \ref{lemma_almost_straight}. Note that $\tilde{\gamma}$ is Hausdorff close to $\gamma\setminus K^\sharp$. By Proposition \ref{lemma_neckconnectedsum} there exists a suitable isotopy between $K$ and $\mathcal{G}_{r_s}(K^\sharp,\tilde{\gamma})$.

Let $\hat{K}^\sharp$ be a connected component of $K^\sharp$.

If one of the caps of $\hat{K}^\sharp$ is a $(C,\eps)$-cap as in (2a) of Definition \ref{def_C_cap}, then $\hat{K}^\sharp$ is convex by property (3) of the surgery (see Definition \ref{def_surgery}). It can thus be deformed to a round ball by Proposition \ref{lemma_convexcontrolled}.

If both caps of $\hat{K}^\sharp$ are $\eps$-close to standard caps (see Definition \ref{def_stdcap}), then our domain $\hat{K}^\sharp$ is a small perturbation of a capped-off cylinder (see Definition \ref{cap_off_def}), and thus can be deformed monotonically to a slightly smaller capped-off cylinder. Letting the capped-off cylinder flow by mean curvature, it will instantaneously become strictly convex, so it is certainly isotopic to a round ball by Proposition \ref{lemma_convexcontrolled}.   

Let $\{K^{\sharp}_t\}_{t\in[0,1]}$ be the union of the above isotopies between the connected components of $K^{\sharp}$ and balls. Letting $r_{\min}$ be the smallest radius among the radii of the balls of $K^{\sharp}$, let $\{K^{\sharp}_t\}_{t\in[1,2]}$ be an isotopy that concatenates smoothly at $t=1$ and shrinks all balls further to balls of radius $r_{\min}/10$. Let $\{\tilde{\gamma}_t\}_{t\in[0,2]}$ be the family of curves which follows $K^{\sharp}_t$ by normal motion starting at $\tilde{\gamma}_0=\tilde{\gamma}$. Then $\{\mathcal{G}_{r_s}(K^\sharp_t,\tilde{\gamma}_t)\}_{t\in [0,2]}$ provides the last step of the asserted isotopy. 
\end{proof}

\section{Mean curvature flow with surgery}\label{sec_mcfsurgery}

In this section, we recall the relevant facts about mean curvature flow with surgery from \cite{HK_surgery} that we need for the proof of the main theorem. We start with the more flexible notion of $(\alpha,\delta)$-flows.

\begin{definition}[$(\alpha,\delta)$-flow {\cite[Def. 1.1]{HK_surgery}}]\label{def_alphadelta}
An \emph{$(\alpha,\delta)$-flow} $\K$ is a collection of finitely many smooth $\al$-noncollapsed mean curvature flows $\{K_t^i\subseteq \R^{n+1}\}_{t\in[t_{i-1},t_{i}]}$ ($i=1,\ldots,\ell$; $t_0<\ldots< t_\ell$),
such that:
\begin{enumerate}[\hspace{3mm}(1)]
\item for each $i=1,\ldots,\ell-1$, the final time slices of some collection of disjoint strong $\delta$-necks
 are replaced by pairs of standard caps as described in Definition \ref{def_surgery},
 giving a domain $K^\sharp_{t_{i}}\subseteq K^{i}_{t_{i}}=:K^-_{t_{i}}$.
\item the initial time slice of the next flow, $K^{i+1}_{t_{i}}=:K^+_{t_{i}}$, is obtained from $K^\sharp_{t_{i}}$ by discarding some connected components.
\item all surgeries are at comparable scales, i.e. there exists a radius $r_\sharp=r_\sharp(\K)>0$, such that all necks in (1) have radius $r\in[\tfrac{1}{2}r_\sharp,2 r_\sharp]$.
\end{enumerate}
\end{definition}

\begin{remark} To avoid confusion, we emphasize that the word `some' allows for the empty set,
i.e. some of the inclusions $K_{t_i}^+\subseteq K_{t_i}^\sharp\subseteq K_{t_i}^-$ could actually be equalities. In other words, there can be some times $t_i$ where effectively only one of the steps (1) or (2) is carried out.
Also, the flow can become extinct, i.e. we allow the possibility that $K^{i+1}_{t_{i}}=\emptyset$.
\end{remark}

A mean curvature flow with surgery is an $(\al,\de)$-flow subject to additional conditions. Besides the neck-quality $\delta>0$, in a flow with surgery we have three curvature-scales $H_{\textrm{trig}}> H_{\textrm{neck}} > H_{\textrm{th}} > 1$, called the trigger-, neck- and thick-curvature. The flow is defined for every smooth compact $2$-convex initial condition.
The following definition quantifies the relevant parameters of the initial domain.

\begin{definition}[{Controlled initial condition \cite[Def. 1.15]{HK_surgery}}]\label{def_initialdata}
Let $\Balpha=(\al,\beta,\gamma)\in(0,n-1)\times (0,\tfrac{1}{n-1})\times (0,\infty)$.
A smooth compact domain $K_0\subset \R^N$ is called an \emph{$\Balpha$-controlled initial condition},
if it is $\alpha$-noncollapsed and satisfies the inequalities $\lambda_1+\lambda_2\geq \beta H$ and $H\leq \gamma$.
\end{definition}

The definition of mean curvature flow with surgery is as follows.

\begin{definition}[Mean curvature flow with surgery {\cite[Def. 1.17]{HK_surgery}}]\label{def_MCF_surgery}
An \emph{$(\Balpha,\de,\mathbb{H})$-flow},  where $\mathbb{H}=(H_{\textrm{th}},H_{\textrm{neck}},H_{\textrm{trig}})$, is an $(\al,\de)$-flow $\{K_t\subset \R^N\}_{t\geq 0}$ (see Definition \ref{def_alphadelta}) with $\lambda_1+\lambda_2\geq \beta H$, and
with $\Balpha=(\al,\beta,\gamma)$-controlled initial condition $K_0\subset \R^N$ (see Definition \ref{def_initialdata}) such that
\begin{enumerate}[\hspace{3mm}(1)]
\item $H\leq H_{\textrm{trig}}$ everywhere, and
surgery and/or discarding occurs precisely at times $t$ when $H=H_{\textrm{trig}}$ somewhere.
\item The collection of necks in item (1) of Definition \ref{def_alphadelta} is a minimal collection of solid $\de$-necks of curvature $H_{\textrm{neck}}$ which
separate the set $\{H=H_{\textrm{trig}}\}$ from $\{H\leq H_{\textrm{th}}\}$ in the domain
$K_t^-$.
\item $K_t^+$ is obtained from $K_t^\sharp$ by discarding precisely those connected components with $H>H_{\textrm{th}}$ everywhere.
In particular, of each pair of facing surgery caps precisely one is discarded. 
\item If a strong $\delta$-neck from (2) also is a strong $\hat{\delta}$-neck for some $\hat{\delta}<\delta$, then property (4) of Definition \ref{def_surgery} also holds with $\hat{\delta}$ instead of $\delta$.
\end{enumerate}
\end{definition}

\begin{remark}\label{rem_fix_std_cap}
Strictly speaking, a flow with surgery also depends on the choice of a suitable standard cap (see Definition \ref{def_stdcap}) and a suitable choice of the cap separation parameter $\Gamma$, see \cite[Convention 2.11]{HK_surgery}.
\end{remark}

Having discussed the precise definition of mean curvature flow with surgery, we can now recall the existence theorem.

\begin{theorem}[{Existence of MCF with surgery \cite[Thm. 1.21]{HK_surgery}}]\label{thm_main_existence}
There are constants $\ol{\de}=\ol{\de}(\Balpha)>0$ and $\Theta(\de)=\Theta(\Balpha,\de)<\infty$ ($\delta\leq\bar{\de}$) with the following significance.
If $\de\leq\bar{\de}$ and $\mathbb{H}=(H_{\textrm{trig}},H_{\textrm{neck}},H_{\textrm{th}})$ are positive numbers with
${H_{\textrm{trig}}}/{H_{\textrm{neck}}},{H_{\textrm{neck}}}/{H_{\textrm{th}}},H_{\textrm{neck}}\geq \Theta(\de)$,
then there exists an $(\Balpha,\de,\mathbb{H})$-flow $\{K_t\}_{t\in[0,\infty)}$ for every $\Balpha$-controlled initial condition $K_0$.  
\end{theorem}

This existence theorem allows us to evolve any smooth compact $2$-convex initial domain $K_0\subset\mathbb{R}^{n+1}$. By comparison with spheres, the flow of course always becomes extinct in finite time, i.e. there is some $T=T(K_0)<\infty$ such that $K_t=\emptyset$ for all $t\geq T$.  The existence theorem is accompanied by the canonical neighborhood theorem, which gives a precise description of the regions of high curvature.

\begin{theorem}[{Canonical neighborhood theorem \cite[Thm. 1.22]{HK_surgery}}]\label{thm_can_nbd}
For all $\eps>0$ and all $\Balpha$, there exist $\ol{\de}=\ol{\de}(\Balpha)>0$, $H_{\textrm{can}}(\eps)=H_{\textrm{can}}(\Balpha,\eps)<\infty$ and $\Theta_\eps(\delta)=\Theta_\eps(\Balpha,\delta)<\infty$ ($\delta\leq\bar{\de}$) with the following significance.
If $\de\leq\ol{\de}$, and $\K$ is an $(\Balpha,\de,\mathbb{H})$-flow with ${H_{\textrm{trig}}}/{H_{\textrm{neck}}},{H_{\textrm{neck}}}/{H_{\textrm{th}}}\geq \Theta_\eps(\de)$,
then any $(p,t)\in\D \K$ with $H(p,t)\geq H_{\textrm{can}}(\eps)$ is $\eps$-close to either
(a) a $\beta$-uniformly $2$-convex ancient $\al$-noncollapsed flow,
or (b) the evolution of a standard cap preceded by the evolution of a round cylinder.
\end{theorem}

\begin{remark}
The structure of $\beta$-uniformly $2$-convex ancient $\al$-noncollapsed flows and the standard solution is described in {\cite[Sec. 3]{HK_surgery}}.
\end{remark}

In \cite{HK_surgery} the canonical neighborhood theorem (Theorem \ref{thm_can_nbd}) and the theorems about the structure of ancient solutions and the standard solution {\cite[Sec. 3]{HK_surgery}} were combined to obtain topological and geometric information about the discarded components.

\begin{corollary}[{Discarded components \cite[Cor. 1.25]{HK_surgery}}]\label{cor_discarded}
For $\eps>0$ small enough, for any $(\Balpha,\de,\mathbb{H})$-flow with ${H_{\textrm{trig}}}/{H_{\textrm{neck}}},{H_{\textrm{neck}}}/{H_{\textrm{th}}}\geq \Theta_\eps(\de)$ ($\de\leq\bar{\de}$) and
$H_\textrm{th}\geq H_{\textrm{can}}(\eps)$, where $\Theta_\eps(\de)$, $\bar{\de}$ and $H_{\textrm{can}}(\eps)$ are from Theorem \ref{thm_can_nbd}, all discarded components are diffeomorphic to $\bar{D}^N$ or $\bar{D}^{N-1}\times S^1$. Moreover, the components that are diffeomorphic to $\bar{D}^N$ are either (a) convex or (b) a capped $\eps$-tube (see Definition \ref{def_capped_off_chain}).
\end{corollary}

\begin{remark}
The last part of Corollary \ref{cor_discarded} was not explicitly stated in \cite{HK_surgery}, but is contained in the proof given there.
\end{remark}

\section{Proof of the main theorem}\label{sec_proofmain}

\begin{proof}[Proof of the main theorem]
Let $K_0\subset \R^{n+1}$ be a $2$-convex domain diffeomorphic to a ball. By compactness $K_0$ is an $\Balpha$-controlled initial condition for some values $\alpha,\beta,\gamma$ (see Definition \ref{def_initialdata}).
Fix a suitable standard cap $K^{\textrm{st}}=K^{\textrm{st}}(\alpha,\beta)$ and a suitable cap separation parameter $\Gamma$ (see Remark \ref{rem_fix_std_cap}).
Let $\eps>0$ be small enough such that Corollary \ref{cor_discarded}, Proposition \ref{lemma_neckwithstdcaps} and Lemma \ref{lemma_rearrange_curves} apply.
Choose curvature scales $\mathbb{H}$ with
${H_{\textrm{trig}}}/{H_{\textrm{neck}}},{H_{\textrm{neck}}}/{H_{\textrm{th}}},H_{\textrm{neck}}\geq \max\{ \Theta_\eps(\de),\Theta(\delta)\}$ ($\delta\leq\bar{\delta}$)
and $H_\textrm{th}\geq H_{\textrm{can}}(\eps)$, where $\bar{\delta}\leq \eps$ is small enough that both the existence result (Theorem \ref{thm_main_existence}) and the canonical neighborhood property (Theorem \ref{thm_can_nbd}) at any point with $H(p,t)\geq H_{\textrm{can}}(\eps)$ are applicable. Choose the control parameters $(\mathbb{A},b)$ flexible enough and the marble radius $r_m$ and string radius $r_s$ small enough, such that the argument below works.

Consider the evolution $\{K_t\}$ by mean curvature flow with surgery given by Theorem \ref{thm_main_existence} with initial condition $K_0$. Let $0<t_1< \ldots < t_{\ell}$ be the times where there is some surgery and/or discarding (see Definition \ref{def_MCF_surgery}). By the definition of a flow with surgery at each $t_i$ there are finitely many (possibly zero) $\delta$-necks with center $p_i^j$ and radius $r_{\textrm{neck}}=(n-1)H^{-1}_{\textrm{neck}}$ that are replaced by a pair of standard caps. Let $B_i^j:=B_{10\Gamma r_\textrm{neck}}(p_i^j)$, and observe that these balls are pairwise disjoint.\footnote{In fact, they are far away from each other, see \cite[Prop. 2.5]{HK_surgery}.} Similarly, for each discarded component $C_i^j$ Proposition \ref{lemma_neckwithstdcaps} gives a finite collection of $\eps$-neck points, whose centers and radii we denote by $p_{i}^{jk}$ and $r_{i}^{jk}$. Let $B_i^{jk}:=B_{10\Gamma r_i^{jk}}(p_i^{jk})$. The isotopy which we will construct will be monotone outside the set
\begin{equation}\label{exceptional_set}
X:=\bigcup_{i,j}B_i^j\cup \bigcup_{i,j,k}B_i^{jk}\, .
\end{equation}
Note that the balls in \eqref{exceptional_set} are pairwise disjoint.

Let $\A_i$ be the assertion that each connected component of $K^i:=K_{t_{i}}^-$ is isotopic via 2-convex embeddings to a marble tree, with an isotopy which is monotone outside $X$.

\begin{claim}\label{claim_inductionhyp}
All discarded components $C_i^j$ are isotopic via 2-convex embeddings to a marble tree, with an isotopy which is monotone outside $\bigcup_k B_i^{jk}$. In particular, $\A_\ell$ holds.
\end{claim}

\begin{proof}[{Proof of Claim \ref{claim_inductionhyp}}]
Our topological assumption on $K_0$ together with the nature of the surgery process (see Definition \ref{def_MCF_surgery}) implies that all discarded components are diffeomorphic to balls. Thus, by Corollary \ref{cor_discarded}, each discarded component is either (a) convex or (b) a capped $\eps$-tube. Using Proposition \ref{lemma_convexcontrolled} in case (a), respectively Proposition \ref{lemma_neckwithstdcaps} in case (b), we can find a suitable isotopy to a marble tree.

Since $K_{t_{\ell}}^+=\emptyset$, we see that at time $t_{\ell}$ there was only discarding, and no replacement of necks by caps (c.f.  Definition \ref{def_MCF_surgery}). Thus, all connected components of $K^\ell=K_{t_\ell}^-$ get discarded, and $\A_\ell$ holds.
\end{proof}

\begin{claim}\label{claim_inductionstep}
If $0<i<\ell$ and $\A_{i+1}$ holds, so does $\A_i$.
\end{claim}

To prove Claim \ref{claim_inductionstep}, we will also need the following lemma.

\begin{lemma}[moving curves out of surgery regions]\label{lemma_rearrange_curves}
Assume $\{K_t\}_{t\leq 0}$ has a strong $\eps$-neck of radius $r$ and center $p$ at $t=0$, and assume $\{\gamma_t\}_{t\leq t_0}$ is a $b$-controlled curve that meets $\partial K_t$ orthogonally and hits $\partial B_{8\Gamma r}(p)$ at some $t_0\leq 0$. Then, for $\eps$ small enough, there exists a smooth family $\{\tilde{\gamma}^t\}_{t\leq t_0}$ of $\min(b,r)/1000\Gamma$-controlled curves that meets $\partial K_t$ orthogonally such that $\tilde{\gamma}^t$ coincides with $\gamma_t$ for $t\leq-(9\Gamma r)^2$ and outside $B_{9\Gamma r}(p)$, and such that $\partial\tilde{\gamma}^0\notin B_{8\Gamma r}(p)$.
\end{lemma}

\begin{proof}[{Proof of Lemma \ref{lemma_rearrange_curves}}]
Assume without loss of generality that $p=0$ and that the neck is along the $x_1$-axis. If the neck is round, then its radius can be computed as $r_t=\sqrt {r^2-2(n-1)t}$. Note that $\abs{t_0}\leq (8\Gamma r)^2/2(n-1)$ and $\abs{\gamma_{t_0}^1}^2=(8\Gamma r)^2-r_{t_0}^2$.  For general $\eps$-necks, with $\eps$ small enough, this holds up to small error terms. We can then define $\tilde{\gamma}^t$ by sliding the curve along the neck.
\end{proof}

\begin{proof}[{Proof of Claim \ref{claim_inductionstep}}]
Smooth evolution by mean curvature flow provides a monotone isotopy between $K_{t_i}^+$ and $K^{i+1}$. Recall that $K_{t_{i}}^+\subseteq K_{t_{i}}^\sharp\subseteq K_{t_{i}}^-=K^i$ is obtained by performing surgery on a minimal collection of disjoint $\delta$-necks separating the thick part and the trigger part and/or discarding connected components that are entirely covered by canonical neighborhoods.

By induction hypothesis the connected components of $K^{i+1}$ are isotopic to marble trees, and by Claim \ref{claim_inductionhyp} the discarded components are isotopic to marble trees as well. It follows that all components of $K_{t_{i}}^\sharp$ are isotopic to marble trees. Let $\{L_t\}_{t\in [0,1]}$ denote such an isotopy deforming $L_0=K_{t_{i}}^\sharp$ into a union of marble trees $L_1$, which is monotone outside $X$. We now want to glue together the isotopies of the components. If there was only discarding at time $t_{i}$ there is no need to glue. Assume now $L_0$ has at least two components.

For each surgery neck at time $t_i$, select an almost straight line $\gamma_i^j$ between the tips of the corresponding pair of standard caps in $K_{t_{i}}^\sharp$ as in Lemma \ref{lemma_almost_straight}. Let $\gamma_0=\bigcup_{j}\gamma_i^j$.
By Proposition \ref{lemma_neckconnectedsum} the domain $K^i=K_{t_i}^-$ is isotopic to $\mathcal{G}_{r_s}(K_{t_{i}}^\sharp,\gamma_0)$ via 2-convex embeddings, with an isotopy that is trivial outside $X$. Finally, to get an isotopy $\mathcal{G}_{r_s}(L_t,\gamma_t)$ it remains to construct a suitable family of curves $\{\gamma_t\}_{t\in [0,1]}$ along which we can do the gluing. Start with $\gamma_0$. Essentially we define $\gamma_t$ by following the points where $\gamma_t$ touches $\partial L_t$ via normal motion. It can happen at finitely many times $t$ that $\gamma_{t}$ hits $\partial X$. In that case, we modify $\gamma_t$ according to Lemma \ref{lemma_rearrange_curves}, and then continue via normal motion. Then $\mathcal{G}_{r_s}(L_t,\gamma_t)_{t\in [0,1]}$ gives the last bit of the desired isotopy.
\end{proof}

It follows from backwards induction on $i$, that $\A_1$ holds. Note that smooth mean curvature flow provides a monotone isotopy via 2-convex embeddings between $K^1$ and $K_0$. In particular, $K^1$ has only one connected component. Finally, we can apply the theorem about marble trees (Theorem \ref{thm_marble_tree}) to deform the marble tree isotopic to $K^1$ into a round ball. Thus, any 2-convex embedded sphere can be smoothly deformed, through 2-convex embeddings, into a round sphere. This proves that $\mathcal{M}_n^{\textrm{2-conv}}$ is path-connected.
\end{proof}

\bibliography{buzanohaslhoferhershkovits}

\newcommand{\noopsort}[1]{} \newcommand{\singleletter}[1]{#1}
\begin{thebibliography}{10}

\bibitem{Alexander}
J.~Alexander.
\newblock On the subdivision of $3$-space by a polyhedron.
\newblock {\em P.N.A.S.}, 10:6--8, 1924.

\bibitem{ref_neckpinch}
S.~Altschuler, S.~Angenent, and Y.~Giga.
\newblock Mean curvature flow through singularities for surfaces of rotation.
\newblock {\em J. Geom. Anal.}, 5(3):293--358, 1995.

\bibitem{andrews1}
B.~Andrews.
\newblock Noncollapsing in mean-convex mean curvature flow.
\newblock {\em Geom. Topol.}, 16(3):1413--1418, 2012.

\bibitem{ref_degneckpinch}
S.~Angenent and J.~Vel{\'a}zquez.
\newblock Degenerate neckpinches in mean curvature flow.
\newblock {\em J. Reine Angew. Math.}, 482:15--66, 1997.

\bibitem{BrendleHuisken}
S.~Brendle and G.~Huisken.
\newblock Mean curvature flow with surgery of mean convex surfaces in {$\Bbb
  R^3$}.
\newblock {\em Invent. Math.}, 203(2):615--654, 2016.

\bibitem{BrendleSchoen}
S.~Brendle and R.~Schoen.
\newblock Manifolds with {$1/4$}-pinched curvature are space forms.
\newblock {\em J. Amer. Math. Soc.}, 22(1):287--307, 2009.

\bibitem{Carr}
R.~Carr.
\newblock Construction of manifolds of positive scalar curvature.
\newblock {\em Trans. Amer. Math. Soc.}, 307(1):63--74, 1988.

\bibitem{Cerf}
J.~Cerf.
\newblock Sur les diffeomorphismes de la sphere de dimension trois
  {$(\Gamma_4=0)$}.
\newblock {\em Springer Lecture notes}, 53, 1968.

\bibitem{CrowleySchick}
D.~Crowley and T.~Schick.
\newblock The {G}romoll filtration, {$KO$}-characteristic classes and metrics
  of positive scalar curvature.
\newblock {\em Geom. Topol.}, 17(3):1773--1789, 2013.

\bibitem{Grayson}
M.~Grayson.
\newblock The heat equation shrinks embedded plane curves to round points.
\newblock {\em J. Differential Geom.}, 26(2):285--314, 1987.

\bibitem{GL}
M.~Gromov and B.~Lawson.
\newblock The classification of simply connected manifolds of positive scalar
  curvature.
\newblock {\em Ann. of Math. (2)}, 111(3):423--434, 1980.

\bibitem{Haslhofer_bowl}
R.~Haslhofer.
\newblock Uniqueness of the bowl soliton.
\newblock {\em Geom. Topol.}, 19(4):2393--2406, 2015.

\bibitem{HK}
R.~Haslhofer and B.~Kleiner.
\newblock Mean curvature flow of mean convex hypersurfaces.
\newblock {\em Comm. Pure Appl. Math.}, 70(3):511--546, 2017.

\bibitem{HK_surgery}
R.~Haslhofer and B.~Kleiner.
\newblock Mean curvature flow with surgery.
\newblock {\em Duke Math. J.}, 166(9):1591--1626, 2017.

\bibitem{Hatcher}
A.~Hatcher.
\newblock A proof of the {S}male conjecture, {${\rm Diff}(S^{3})\simeq {\rm
  O}(4)$}.
\newblock {\em Ann. of Math. (2)}, 117(3):553--607, 1983.

\bibitem{Huisken_local_global}
G.~Huisken.
\newblock Local and global behaviour of hypersurfaces moving by mean curvature.
\newblock In {\em Differential geometry: partial differential equations on
  manifolds ({L}os {A}ngeles, {CA}, 1990)}, volume~54 of {\em Proc. Sympos.
  Pure Math.}, pages 175--191. Amer. Math. Soc., Providence, RI, 1993.

\bibitem{HuiskenSinestrari}
G.~Huisken and C.~Sinestrari.
\newblock Mean curvature flow with surgeries of two-convex hypersurfaces.
\newblock {\em Invent. Math.}, 175(1):137--221, 2009.

\bibitem{HuiskenSinestrari_convexity}
G.~Huisken and C.~Sinestrari.
\newblock Convexity estimates for mean curvature flow and singularities of mean
  convex surfaces.
\newblock {\em Acta Math.}, 183(1):45--70, {\noopsort{b}}1999.

\bibitem{KKM}
N.~Kapouleas, S.~Kleene, and N.~M{\o}ller.
\newblock Mean curvature self-shrinkers of high genus: non-compact examples.
\newblock {\em arXiv:1106.5454}, 2011.

\bibitem{KreckStolz}
M.~Kreck and S.~Stolz.
\newblock Nonconnected moduli spaces of positive sectional curvature metrics.
\newblock {\em J. Amer. Math. Soc.}, 6(4):825--850, 1993.

\bibitem{Marques}
F.~C. Marques.
\newblock Deforming three-manifolds with positive scalar curvature.
\newblock {\em Ann. of Math. (2)}, 176(2):815--863, 2012.

\bibitem{Milnor}
J.~Milnor.
\newblock On manifolds homeomorphic to the {$7$}-sphere.
\newblock {\em Ann. of Math. (2)}, 64:399--405, 1956.

\bibitem{Munkres}
J.~Munkres.
\newblock Differentiable isotopies on the {$2$}-sphere.
\newblock {\em Michigan Math. J.}, 7:193--197, 1960.

\bibitem{perelman_entropy}
G.~Perelman.
\newblock The entropy formula for the {R}icci flow and its geometric
  applications.
\newblock {\em arXiv:math/0211159}, 2002.

\bibitem{perelman_surgery}
G.~Perelman.
\newblock Ricci flow with surgery on three-manifolds.
\newblock {\em arXiv:math/0303109}, 2003.

\bibitem{SY}
R.~Schoen and S.~T. Yau.
\newblock On the structure of manifolds with positive scalar curvature.
\newblock {\em Manuscripta Math.}, 28(1-3):159--183, 1979.

\bibitem{sheng_wang}
W.~Sheng and X.~Wang.
\newblock Singularity profile in the mean curvature flow.
\newblock {\em Methods Appl. Anal.}, 16(2):139--155, 2009.

\bibitem{Smale}
S.~Smale.
\newblock Diffeomorphisms of the {$2$}-sphere.
\newblock {\em Proc. Amer. Math. Soc.}, 10:621--626, 1959.

\bibitem{Weyl}
H.~Weyl.
\newblock {\"U}ber die {B}estimmung einer geschlossenen konvexen {F}laeche
  durch ihr {L}inienelement.
\newblock {\em Vierteljahrsschr. Naturforsch. Ges. Zur.}, 61:40--72, 1916.

\end{thebibliography}

\bibliographystyle{abbrv}

\vspace{10mm}
Reto Buzano (M\"{u}ller): r.buzano@qmul.ac.uk\\
{\sc School of Mathematical Sciences, Queen Mary University of London, Mile End Road, London E1 4NS, UK}\\

Robert Haslhofer: roberth@math.toronto.edu\\
{\sc Department of Mathematics, University of Toronto, 40 St George Street, Toronto, ON M5S 2E4, Canada}\\

Or Hershkovits: orher@stanford.edu\\
{\sc Department of Mathematics, Stanford University, 450 Serra Mall, Stanford, CA 94305, USA}\\

\end{document}